\documentclass[english, 11pt, a4paper]{article}
\PassOptionsToPackage{dvipsnames}{xcolor}
\usepackage{babel}
\usepackage[utf8]{inputenc}
\usepackage[T1]{fontenc}
\usepackage[top=2.5cm, bottom=2.5cm, left=2.5cm, right=2.5cm]{geometry} 
\usepackage{verbatim}
\usepackage{amsmath}
\usepackage{amsthm}
\usepackage{amsfonts}
\usepackage{mathabx}
\usepackage{color}
\usepackage{xcolor}
\usepackage{amssymb}
\usepackage{mathrsfs}
\usepackage{subfigure}
\usepackage{float}
\usepackage[most]{tcolorbox}
\usepackage[breaklinks=true]{hyperref}
\usepackage{enumitem}
\usepackage{bbm}
\usepackage{xspace}
\usepackage{regexpatch}
\usepackage{colortbl}%
\usepackage{enumitem}
\usepackage{makecell}
\usepackage{microtype}
\usepackage{mathtools,algorithm}
\usepackage[noend]{algpseudocode}

\DeclareMathAlphabet{\mathb}{OML}{cmm}{b}{it}

\newcommand{\N}{\mathbb{N}}		
\newcommand{\R}{\mathbb{R}}		

\newcommand{\syst}[1]{\left \{ \begin{array}{l} #1 \end{array} \right. \kern-\nulldelimiterspace}	
\newcommand{\prox}{\mathrm{prox}}
\newcommand{\argmin}{\mathrm{argmin}}
\newcommand{\argmax}{\mathrm{argmax}}
\newcommand{\dom}{\mathrm{dom}\,}
\DeclareMathOperator{\Proj}{Proj}

\renewcommand{\int}{\mathrm{int}\,}

\newcommand{\argmind}[2]{\ensuremath{\underset{\substack{{#1}}}%
{\mathrm{argmin}}\;\;#2 }}

\newcommand{\lsc}{\Phi_{lsc}}
\newcommand{\proxlsc}{\mathrm{prox}^{\mathrm{lsc},\mathbb{R}}}

\makeatletter
\newlength{\algorithmboxrule}
\newcommand{\algorithmboxcolor}{white}
\xpatchcmd*{\algocf@caption@boxruled}{0.0pt}{2\algorithmboxrule}{}{}
\xpatchcmd*{\algocf@caption@boxruled}{\vrule}{\vrule width \algorithmboxrule}{}{}
\xpatchcmd{\algocf@caption@boxruled}{\hrule}{\hrule height \algorithmboxrule}{}{}
\xpretocmd{\algocf@caption@boxruled}{\color{\algorithmboxcolor}}{}{}
\makeatother

\newtheorem{lemma}{Lemma}[section]
\newtheorem{proposition}[lemma]{Proposition}

\newtheorem{theorem}[lemma]{Theorem}
\theoremstyle{definition}
\newtheorem{definition}[lemma]{Definition}
\newtheorem{remark}[lemma]{Remark}
\newtheorem{example}[lemma]{Example}

\newtheorem{assumption}[lemma]{Assumption}

\title{Proximal Algorithms for a class of abstract convex functions}
\author{ Ewa Bednarczuk\thanks{Warsaw University of Technology,  00-662 Warsaw, Koszykowa 75, Poland} \thanks{Systems Research Institute, PAS,  01-447 Warsaw, Newelska 6, Poland}
\and Dirk Lorenz\thanks{Center for Industrial Mathematics, University of Bremen, Postfach 330440, 28334 Bremen, Germany,}
\and
The Hung Tran\footnotemark[2]}

\begin{document}
\maketitle
\begin{abstract}
  In this paper we analyze a class of nonconvex optimization problem from the viewpoint of abstract convexity. Using the respective generalizations of the subgradient we propose an abstract notion  proximal operator and derive a number of algorithms, namely an abstract proximal point method, an abstract forward-backward method and an abstract projected subgradient method. Global convergence results for all algorithms are discussed and numerical examples are given.
  
\end{abstract}
\vspace{0.4cm} 

\section{Introduction}

In this paper we aim to design a proximal mapping which can work for a specific class of nonconvex functions in the context of abstract convexity. Our goal is to derive proximal algorithms that still exhibit global convergence in this context.

Abstract convex functions have been studied in monographs by Rubinov \cite{Rub2013}, Pallaschke and Rolewicz \cite{Pall2013}, Singer \cite{sin1997} when they explored convexity without linearity. Given a Hilbert space $X$, a function $f: X\to (-\infty,+\infty]$ is convex with respect to the class of functions $\Phi =\{ \phi:X\to \R\}$, or we call $\Phi$-convex, if and only if 
\begin{equation*}
    f(x) = \sup_{\phi\leq f, \phi \in\Phi} \phi (x),
\end{equation*}
for all $x\in X$.
When $\Phi$ is the class of affine functions, then $f$ is lower semicontinuous and convex in the classical sense if and only if $f$ is $\Phi$-convex \cite[Theorem 9.20]{Bau2011}. By allowing the class $\Phi$ to contain nonlinear functions, we obtain a more general concept of convexity, called $\Phi$-convexity.

Various types of abstract convex functions have been discussed e.g. topical and sub-topical functions \cite{rubinov2001topical}, star-shaped functions \cite{rubinov1999}, and positive homogeneous functions \cite{rubinov2002abstract}. Moreover, we can generalize many concepts from convex analysis like conjugation and subdifferentials for solving optimization problems. Jeyakumar \cite{Jey2007} constructed conjugated dual problems for non-affine convex function using abstract conjugation. They also  examined duality between the primal and dual problems, and stated conditions which sum rule for subdifferentials holds. While Burachik \cite{burachik2007} studied duality of constrained problem using augmented Lagrangians built from abstract convexity. They also considered abstract monotone operator and compared them to maximal monotone operator in the classical sense \cite{burachik2008}.

In this work we will focus on the class of $\lsc^\R$-convex functions which covers weakly convex and strongly convex functions.
The framework of abstract convexity includes the class of so-called weakly convex functions which have been useful in applications such as source localization, \cite{beck2008exact}, phase retrieval, \cite{luke2019optimization}, and discrete tomography \cite{schule2005discrete, kadu2019convex} or distributed network optimization \cite{chen2021distributed}.

\subsection{Background and state of the art}
Despite having good theoretical properties, there are not many numerical algorithms for abstract convex optimization problems. Andramonov \cite{andra2002} analyzed cutting plane methods to solve minimization problem of $\Phi$-convex functions and Beliakov \cite{beliakov2003geometry} examined the same algorithm for the class of piecewise linear function. Zhou et al. \cite{zhou2016novel} applied cutting angle method to design a differential evolution algorithm using support hyperplanes as the class $\Phi$.

The elements of the class $\lsc^\R$ are quadratic functions defined on Hilbert space $X$ i.e.
\[
\phi (x) = -a \Vert x\Vert^2 +\langle u,x\rangle +c,
\]
where $u\in X, a,c\in \R$.
The main advantage of choosing this class is that iterative schemes we propose are practically realizable, i.e. we can calculate efficiently $\lsc^\R$-subdifferentials and $\lsc^\R$-conjugate.
Hence, in this work, we aim to design a proximal-based algorithm for the class of abstract convex functions minorized by the set of quadratic functions. Once the proximal mapping is defined, there are possibilities to extend the proximal point method for splitting algorithms like forward-backward or projected subgradient algorithm.

The proximal point method is one of the classical and well-known approach to find the minimizer of a convex lower semicontinuous function. It was first introduced by Moreau \cite{moreau1965proximite} to solve a convex optimization problem by regularization and further studied by Martinet \cite{martinet1970regularisation}, Rockafellar \cite{rockafellar1976monotone}, Brezis and Lions \cite{brezis1978produits} for solving variational inequalities of maximal monotone operators. The proximal point method enjoys nice properties of convergence to the global minima thanks to the basis of subdifferentials and affine minorization \cite{luque1984asymptotic,guler1991convergence,xu2006regularization,combettes2023resolvent}. 
Since the proximal map is defined for non-continuous, non-differentiable functions, it has become a popular tool to solve a wide-range of optimization problems. 
Recently, there are many attempts to make use of proximal mappings for solving nonconvex problems. Kaplan \cite{kaplan1998proximal} investigated the possibilities of proximal mapping for proper lower semicontinuous function $f$ such that $f+\frac{\chi}{2}\Vert \cdot\Vert^2$ is strongly convex with constant $\chi >0$. Hare \cite{hare2009computing} worked on the convergence properties of the class of prox-bounded and lower-$\mathcal{C}^2$ functions, and~\cite{bredies2015minimization} derives a proximal gradient method with a proximal map for non-convex functions.

Without convexity, the standard convex subdifferentials need not to exist. One remedy is to use different concept of subdiferentials, e.g. the most common one is Mordukhovich subdifferentials \cite{mordukhovich1976maximum} which is suitable for general nonconvex functions. However, it is locally defined so one can only have convergence to a critical point. Further convergence results have been made by assuming additional conditions on the function like error bound conditions \cite{solodov2000error} or a Kurdyka-\L ojasiewicz (KL) inequality \cite{attouch2013convergence}, while others focus on specific class of functions e.g. difference of convex functions \cite{mainge2008convergence} or weakly convex functions \cite{davis2019}. In this work, thanks to the definition of $\lsc^\R$-conjugate and $\lsc^\R$-subdifferentials, we can obtain convergence results without relying on additional properties of the minorzed functions. Moreover, as those definitions are globally established, our convergence results are global.


\subsection{Notation}
We consider a Hilbert space $X$ with norm $\left\Vert \cdot\right\Vert _{X}$ and inner product $\left\langle \cdot,\cdot\right\rangle _{X}$. For a nonempty set $C\subset X, \Proj_C(x)$ denotes the projection of $x$ onto $C$. 
We define $\Phi$ as a collection of real-valued functions $\phi:X\to\mathbb{R}$. The domain of a function $f: X\to [-\infty,+\infty]$ is denoted as $\dom f=\{ x\in X: f(x) <+\infty\}$. A function $f$ is proper if $\dom f\neq \emptyset$.
For a set-valued operator $A: X \rightrightarrows Y$, its domain and range are defined as
\begin{equation*}
    \dom A = \left\{ x\in X: Ax \neq \emptyset\right\},\quad \mathrm{ran } A = \left\{ Ax: x\in X\right\},
\end{equation*}
and its inverse is defined by
\begin{align*}
  x\in A^{-1}y \iff y\in Ax.
\end{align*}
It holds that $\dom A^{-1} = \mathrm{ran } A$ and $\dom (A+B) = \dom A \cap \dom B$ for $A,B:X \rightrightarrows Y$ \cite{Bau2011}. We also define $A(\emptyset) = \emptyset$.
With $Id$ we denote the identity mapping. 

\subsection{Outline}
The outline of our paper is as follow: We give formal definition of abstract convex functions with respect to the set of quadratic functions, which we call $\lsc^\R$-convex functions, and define $\lsc^\R$-subdifferentials as well as some examples in Section \ref{sec: preliminary}. In Section \ref{sec: prox operator}, we construct $\lsc^\R$-proximal mapping with respect to the class $\lsc^\R$ and make connection between the fixed point of $\lsc^\R$-operator and global minimiser. We show that the classical proximal operator can be included in $\lsc^\R$-proximal operator.
This new $\lsc^\R$-proximal operator is the main ingredient for $\lsc^\R$-proximal point method, solving global minimization problem in Section \ref{sec: lsc PPA}. Auxiliary convergence results are mentioned which can be applied for all the algorithms in this paper. After $\lsc^\R$-proximal point algorithm ($\lsc^\R$-PPA), we propose $\lsc^\R$-forward-backward algorithm ($\lsc^\R$-FB) in Section \ref{sec: FB algorithm} for the sum of two functions where one is Fr{\'e}chet differentiable with Lipschitz continuous gradient and the second one is $\lsc^\R$-convex. When one function is an indicator function of a closed set, ($\lsc^\R$-FB) algorithm is reduced to $\lsc^\R$-Projected Subgradient algorithm ($\lsc^\R$-PSG). We point out some similarities of ($\lsc^\R$-PSG) when proving convergence with Projected Subgradient algorithm in the convex case (Section \ref{sec: projected sub alg}). Finally, we present some numerical examples in Section \ref{sec: numerical examples} where we apply our $\lsc^\R$-Projected Subgradient algorithm to solve a nonconvex quadratic problem. 

\section{$\lsc^\R$-Convexity and $\lsc^\R$-Subdifferentials}
\label{sec: preliminary}

For a class $\Phi$ of functions of the type $\phi:X\to\mathbb{R}$ we define $\Phi$-convexity as follows.

\begin{definition}\cite{Pall2013,Rub2013}
\label{def:Phi-convex}
A function $f:X\to\left(-\infty,+\infty\right]$ is said to be $\Phi$-convex on $X$
if and only if we can write 
\[
f\left(x\right)=\sup_{\phi\in\Phi,\phi\leq f}\phi\left(x\right),\quad\forall x\in X.
\]
\end{definition}
The functions $\phi\in\Phi$ are called elementary functions. Depending on the choice of the set $\Phi$, we obtain different types of $\Phi$-convex function.
When $\Phi$ is the class of all affine functions, then $f$ is $\Phi$-convex if and only if it is a proper lsc convex function.

We are interested in the following specific class of elementary functions
\[
\lsc^\R=\left\{ \phi:\phi\left(x\right)=-a\left\Vert x\right\Vert ^{2}+\left\langle u,x\right\rangle +c  \text{ where } a,c\in \R,u\in X\right\}.
\]
Notice that $\lsc^\R$ includes the class of affine functions, so a proper lsc convex function is also $\lsc^\R$-convex. In fact, $\lsc^\R$ also covers the class of strongly convex, weakly convex and DC convex functions (see \cite{Pall2013,Rub2013,zalinescu2002convex}). 
For instance, one can define the set of elementary functions
\[
\lsc^a =\left\{ \phi:\phi\left(x\right)=-a\left\Vert x\right\Vert ^{2}+\left\langle u,x\right\rangle +c  \text{ where } c\in \R,u\in X\right\},
\]
by fixing the coefficient $a\in\R$. When $a>0$, $\lsc^a$-convexity is equivalent to weak convexity ($2a$-weakly convex), if $a<0$, $\lsc^a$-convexity is equivalent to strong convexity ($2a$-strongly convex) (see \cite{zalinescu2002convex,vial1983strong}). By \cite[Corollary 11.17]{Bau2011}, any $\lsc^a$-convex function, $a<0$, is supercoercive and has a unique minimiser.
Another example is the sub-class of $\lsc^\R$, where one considers the coefficients $a\geq 0$ i.e.
\begin{equation*}
    \lsc^{\geq}=\left\{ \phi:\phi\left(x\right)=-a\left\Vert x\right\Vert ^{2}+\left\langle u,x\right\rangle +c  \text{ where } a \geq 0,c\in \R,u\in X\right\}.
\end{equation*}
which has been studied extensively in \cite{Rub2013} and further investigated in \cite{Bed2020}.
In fact, the class of $\lsc^{\geq}$-convex functions has been proved to coincide with the set of all lower semi-continuous functions minorized by a function from $\lsc^{\geq}$ on $X$ \cite[Proposition 6.3]{Rub2013}.

\begin{proposition}
\label{proP:121}
Let $X$ be a Hilbert space, $f:X\to (-\infty,+\infty]$ be proper. We have ${\lsc^{\geq} \subset \lsc^\R }$. If $f$ is $\lsc^\R$-convex and there exists $\phi\in\lsc^\R$ with $a_\phi <0$ and
\begin{equation}
\exists x\in \dom f, f(x)=\phi (x) \text{ and } f(y) \geq \phi (y), \ \forall y\in X, \label{eq: phi in lsc R}
\end{equation}
then $\lim_{\Vert x\Vert \to +\infty} f(x) = +\infty$  and there exists $\psi\in\lsc^{\geq}$ such that
\begin{equation}
f(x) =\psi (x), \text{ and } f(y)\geq \psi (y), \ \forall y\in X, \label{eq: psi in phi lsc}    
\end{equation}
which implies $f$ is $\lsc^{\geq}$-convex.
\end{proposition}
\begin{proof}
By assumption \eqref{eq: phi in lsc R}, we know that 
\begin{equation}
f(y) \geq \phi (y) = -a_\phi \Vert y\Vert^2 +\langle u_\phi, y\rangle +c_\phi, \label{eq: f >= phi lsc R}
\end{equation}
for all $y\in X$ with $a_\phi <0$. By taking the limit both sides of \eqref{eq: f >= phi lsc R} with $\Vert y\Vert \to +\infty$, we have
\begin{equation}
    \lim_{\Vert y\Vert \to +\infty} f(y) \geq \lim_{\Vert y\Vert \to +\infty} \phi (y) = +\infty.
\end{equation}
Now, we want to find $\psi\in\lsc^{\geq}$ such that \eqref{eq: psi in phi lsc} holds. For simplicity, we can find $\psi\in\lsc^{\geq}$ with the form
\[
\psi (y) = \langle u_\psi,y\rangle + c_\psi, \ \phi(y) \geq \psi(y), \ \forall y\in X, \text{ and } \phi(x)=\psi (x).
\]
This means we need to find $u_\psi,c_\psi$ such that 
\begin{equation*}
    h(y) := \phi(y)-\psi(y)\geq 0 \ \forall y\in X, \text{ and } h(x) =0.
\end{equation*}
By solving the following system
\begin{align}
    -a_\phi \Vert y\Vert^2 +\langle u_\phi, y\rangle +c_\phi &\geq  \langle u_\psi, y\rangle +c_\psi  \\
    -2a_\phi x +u_\phi -u_\psi =0,
\end{align}
which gives us $u_\psi = -2a_\phi x +u_\phi$ and $c_\psi = a_\phi \Vert x\Vert^2 +c_\phi$.
\end{proof}

With the definition of $\lsc^\R$-convexity, one can provide the definition of $\lsc^\R$-conjugate and $\lsc^\R$-subdifferentials in a similar way as it is done in convex analysis.
In the following, we present the definition of $\lsc^\R$-subddifferentials.

\begin{definition}\cite{Pall2013,Rub2013}
\label{def:Phi-subdiff}
The $\lsc^\R$-subgradient of $f$ at $x_{0}\in\mathrm{dom }f$ is an element $\phi\in\lsc^\R$ such that
\begin{align}
\label{eq: def subdiff}
\left(\forall y\in X\right)\quad f\left(y\right)-f\left(x_{0}\right) & \geq\phi\left(y\right)-\phi\left(x_{0}\right).
\end{align}
The collection of all such $\phi$ satisfying \eqref{eq: def subdiff} is called $\lsc^\mathbb{R}$-subdifferential and is denoted by $\partial_{lsc}^\R f(x_0)$.
Analogously, we define $\lsc^{\geq}$-subdifferentials and $\lsc^{a}$-subdifferentials of $f$ and denote as $\partial_{lsc}^{\geq} f$ and $\partial_{lsc}^a f$, respectively. In particular, $\partial_{lsc}^0 f$ denotes the subdifferentials in the sense of convex analysis.
\end{definition}

Clearly, $\partial_{lsc}^\R  f$ is a set-valued mapping $\partial_{lsc}^\R  f: X \rightrightarrows \lsc^\R $ and $\dom\partial_{lsc}^\R  f = \{ x\in X: \partial_{lsc}^\R  f(x) \neq \emptyset\}$.
If $x\notin \dom f$ then $\partial_{lsc}^\R  f(x) =\emptyset$. For an $\lsc^\R$-convex function $f$ and $x_0\in\dom f$, $\partial_{lsc}^\R f(x_0) \neq \emptyset$ if there exists $\phi\in\lsc^\R, \phi \leq f$ and $\phi (x_0) = f(x_0)$. 

By Proposition \ref{proP:121}, we see that the class of $\lsc^\R$-convex functions coincides with the class of $\lsc^{\geq}$-convex functions. On the other hand, the set $\lsc^\R$ is larger than $\lsc^{\geq}$ which implies $\lsc^\R$-subdifferentials can be larger than $\lsc^{\geq}$-subdifferentials. This is the reason why we decide to use the class $\lsc^\R$-convex functions in the sequel.

\begin{figure}[ht]
    \centering
    \includegraphics[scale=0.8]{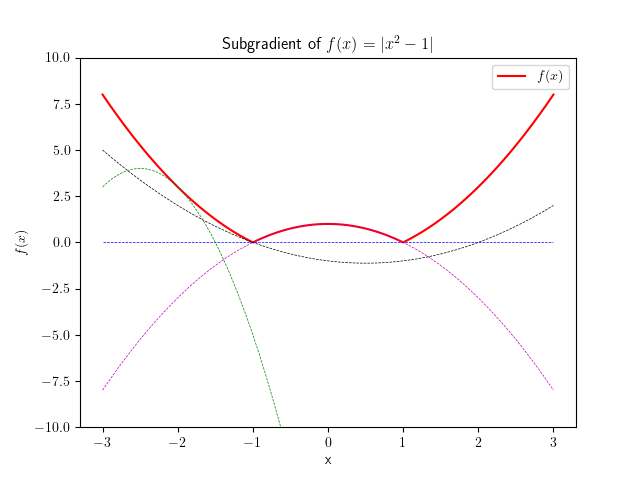}
    \caption{$\lsc^\R$-Subgradient of $f$ at different points}
    \label{fig:subgrad plot}
\end{figure}


\begin{remark}
\
\begin{enumerate}
    \item From the definition of $\lsc^\R$-subdifferentials, the constant $c$ cancels in \eqref{def:Phi-subdiff}. Therefore, for $x_0\in \dom f, \phi\in \partial_{lsc}^\R f(x_0)$ and any real $c\in\R$, the function  
    \begin{equation*}
        \phi (\cdot) -c = -a \Vert \cdot\Vert^2 +\langle u,\cdot \rangle \in \partial_{lsc}^\R f(x_0).
    \end{equation*}
It is shown in \cite{Bed2020} that the constant $c$ is not important in the study of conjugate duality and can be neglected . 

\item The above remark does not imply that the constant term is insignificant. In fact, its purpose is to serve as a connection between $\lsc^\R$-convexity and $\lsc^\R$-subdifferentiability. Let us consider $x_0\in\dom f$ and $\phi\in\partial_{lsc}^\R f(x_0)$ 
\begin{equation*}
    (\forall y\in X) \quad f(y) -f(x_0) \geq \phi (y) -\phi (x_0).
\end{equation*}
The function $h(x) := \phi(x) - \phi(x_0) +f(x_0) \leq f(x)$ for all $x\in X$ and $f(x_0) = h(x_0)$. Hence, $h$ is a $\lsc^\R$-minorant of $f$, i.e. $h\in\lsc^\R$ with the constant term $-a \Vert x_0\Vert^2 +\langle u,x_0 \rangle +f(x_0)$. Therefore, the function $f$ is $\lsc^\R$-convex on $X$ if the domain of $\lsc^\R$-subdifferentials is the whole space $X$. 
\item We point out that the $\lsc^\R$-subdifferentiability of $f$ on $X$ implies $\lsc^\R$-convexity of $f$ on $X$  but not conversely: Consider the function $f(x) = -\Vert x \Vert$, it is $\lsc^\R$-convex but its $\lsc^\R$-subdifferentials is empty at $x=0$.

\item In general, for any class $\Phi$, we have $f\in \partial_{\Phi} f(x)$ for all $x\in X$ if $f\in \Phi$. 
Even if $f\in\Phi$ is differentiable and $\Phi$-convex then $\partial_\Phi$ is still a set (in convex analysis, subdifferentials is unique and coincide with the gradient when the function is differential and convex, see Lemma \ref{lem: lsc subdiff and grad} below).
\end{enumerate}
\end{remark}

Since we mostly deal with $\lsc^\R$-subdifferentials, in the sequel, we identify elements $\phi\in\lsc^{\R}$ with pairs $\left(a,u\right)$ where $a\in \R ,u\in X$.

We demonstrate some examples of $\lsc^\R$-subdifferentials which will be important in the following sections.

\begin{example}
\label{ex:1 J function}
Let $\gamma>0$ and consider the function $g_\gamma \left(x\right)=\frac{1}{2\gamma}\left\Vert x\right\Vert ^{2}$. We calculate the $\Phi_{lsc}^\R$-subgradient of $g_\gamma$ at $x_0\in X$.
By definition, $\phi=(a,u)\in \partial_{lsc}^\R g_\gamma(x_0)$ has to satisfy
\begin{align*}
    g_\gamma (y)-g_\gamma (x_0) \geq \phi (y) -\phi (x_0),
\end{align*}
for all $y\in X$.
Simplifying both sides gives
\begin{align}
    \label{eq: ex1 subgrad def}
   \left(\frac{1}{2\gamma}+a\right) \left\Vert y\right\Vert ^{2} -\langle u,y\rangle\geq  \left(\frac{1}{2\gamma}+a\right) \left\Vert x_0\right\Vert ^{2}- \langle u, x_0\rangle.
\end{align}
\begin{itemize}
\item If $a=-1/(2\gamma)$, then $\langle u,y-x\rangle \leq 0$ for all $y\in X$, by \eqref{eq: ex1 subgrad def} $u$ musts be zero. 

\item If $a > -1 /(2\gamma)$, from \eqref{eq: ex1 subgrad def} we have
\begin{equation}
    \label{eq: ex1 case 2}
    (\forall y\in X) \qquad \left(\frac{1}{2\gamma}+a\right) \left\Vert y - \frac{u}{\frac{1}{\gamma} +2a}\right\Vert ^{2} \geq  \left(\frac{1}{2\gamma}+a\right) \left\Vert x_0 - \frac{u}{\frac{1}{\gamma} +2a}\right\Vert ^{2}.
\end{equation}
This means $u= \left(\frac{1}{\gamma}+2a\right) x_0$.
In particular, $\partial_{lsc}^a g_\gamma (x_0) = \{\phi\in\lsc^a: u = \left(\frac{1}{\gamma}+2a\right) x_0\}$.

\item If $a<-1 /(2\gamma)$, from \eqref{eq: ex1 case 2} it should be
\begin{equation}
    (\forall y\in X) \qquad \left\Vert y - \frac{u}{\frac{1}{\gamma} +2a}\right\Vert ^{2} \leq  \left\Vert x_0 - \frac{u}{\frac{1}{\gamma} +2a}\right\Vert ^{2},\nonumber
\end{equation}
which is impossible. 
\end{itemize}
Hence, the $\lsc^\R$-subddifferentials of $g_\gamma$ at $x_0$ takes the form
\begin{equation}
\partial_{lsc}^\R g_\gamma \left(x_{0}\right)=\left\{ \phi\in\lsc^\R: \phi =\left(a,\left(\frac{1}{\gamma}+2a\right)x_{0}\right),\ 2\gamma a\geq -1\right\},
\label{eq: J_gamma form}
\end{equation}
and similarly, we calculate
\[
\partial_{lsc}^{\geq} g_\gamma \left(x_{0}\right)=\left\{ \phi\in\lsc^\R: \phi =\left(a,\left(\frac{1}{\gamma}+2a\right)x_{0}\right)\right\}.
\]
\end{example}
Notice that $g_\gamma \in \partial_{lsc}^\R g_\gamma (x_0)$ for any $x_0\in X$ as we can write $g_\gamma = \left(-\frac{1}{2\gamma},0\right) \in\lsc^\R$.

Now for any $\phi\in\lsc^\R$, we calculate 
\begin{equation*}
\left(\partial_{lsc}^\R g_\gamma \right)^{-1}\left(\phi\right) =\left\{ x_0\in X:\phi \in \partial_{lsc}^\R g_\gamma (x_0)\right\}.
\end{equation*}
Consider the following cases
\begin{itemize}
\item If $2\gamma a < -1$, then $\left(\partial_{lsc}^\R g_\gamma \right)^{-1}\left(\phi\right) = \emptyset$ as we need $2\gamma a \geq -1$ as in Example \ref{ex:1 J function}. 

\item If $2\gamma a > -1$, for $\phi\in\partial_{lsc}^\R g_\gamma (x_0)$ must take the form $\phi = (a,u) = \left(a,\left(\frac{1}{\gamma}+2a\right)x_{0} \right)$. We can find $x_0$ by letting $u = \left(\frac{1}{\gamma}+2a\right)x_{0}$ i.e.
\begin{equation*}
 x_0 = \frac{1}{\frac{1}{\gamma}+2a}u=\frac{\gamma}{1+2\gamma a}u.
\end{equation*}

\item If $a=-\frac{1}{2\gamma}$, then $u$ must be zero and $\left(-\frac{1}{2\gamma},0\right)\in \partial_{lsc}^\R g_\gamma(x)$ for all $x\in X$. 
\end{itemize}
In conclusion, for any $\phi\in \lsc^\R$, we have
\begin{align}
\label{eq: subgrad of J_gamma}
\left(\partial_{lsc}^\R g_\gamma \right)^{-1}\left(\phi\right) = \begin{cases}
\left\{\frac{\gamma}{1+2\gamma a}u\right\} & \text{ if } 2\gamma a > -1\\
X & \text{ if } a=-\frac{1}{2\gamma}, u=0 \\
\emptyset & \text{ otherwise }\\
\end{cases}.
\end{align}

Observe that 
\begin{equation}
\label{eq: identity of J_gamma}
\phi \in \partial_{lsc}^\R g_\gamma\left( (\partial_{lsc}^\R g_\gamma) ^{-1} (\phi)\right),
\end{equation}
for all $\phi \in \dom (\partial_{lsc}^\R g_\gamma)^{-1}$. Clearly, when $a=-1/(2\gamma),$ $u=0$, \eqref{eq: identity of J_gamma} holds trivially. When $\phi=(a,u)\in \lsc^\R$ with $2\gamma a >-1$, then 
\begin{equation*}
    \partial_{lsc}^\R g_\gamma \left((\partial_{lsc}^\R g_\gamma)^{-1} (\phi)\right) = \partial_{lsc}^\R g_\gamma \left( \frac{\gamma u}{1+2\gamma a}\right) = \left\{ \phi_1\in\lsc^\R : \phi_1 = \left(a_1,\frac{1+2\gamma a_1}{1+2\gamma a}u \right), a_1 \geq -\frac{1}{2\gamma}\right\}.
\end{equation*}
Since $a_1 \geq -1/2\gamma$, it is clear that $\phi$ lies inside the above set. Hence, \eqref{eq: identity of J_gamma} holds.

\begin{example}
\label{ex: xQx}
Let $X = \mathbb{R}^{n}$, consider the function $f\left(x\right)=\left\langle x,Qx\right\rangle $ where
$Q\in\mathbb{R}^{n\times n}$ is a real symmetric matrix. We compute the $\lsc^\R$-subdifferential of $f$ at $x\in\R^n$. We need to find $\phi\in\lsc^\R$ such that $\phi = (a,u)\in\partial_{lsc}^\R f\left(x\right)$. By definition, we have
\begin{equation}
(\forall y\in \R^n) \qquad \left\langle y,\left(Q+aId\right)y\right\rangle -\left\langle x,\left(Q+aId\right)x\right\rangle  \geq\left\langle u,y-x\right\rangle .\label{eq:example_1}
\end{equation}
For any matrix $A\in \R^{n\times n}$, the following holds
\begin{align*}
\left\langle y,Ay\right\rangle -\left\langle x,Ax\right\rangle  & =\left\langle y-x,A\left(y-x\right)\right\rangle +\left\langle y-x,Ax\right\rangle +\left\langle y,A^{\top}x\right\rangle -\left\langle x,Ax\right\rangle \\
 & =\left\langle y-x,A\left(y-x\right)\right\rangle +\left\langle y-x,\left(A+A^{\top}\right)x\right\rangle ,
\end{align*}
where $A^{\top}$ is the transpose of $A$. 
Since $Q$ is real symmetric, we apply the above identity for $A =\left(Q+aId\right)$ which is also real symmetric i.e. $A^{\top} =A$.
Then inequality \eqref{eq:example_1} becomes 
\begin{equation}
\left\langle y-x,\left(Q+aId\right)\left(y-x\right)\right\rangle \geq\left\langle u-2\left(Q+aId\right) x,y-x\right\rangle.\label{eq:example_1 1}
\end{equation}
We can diagonalize $(Q+a Id)$ into $P(D+a Id)P^\top$ with a diagonal matrix $D$ with the eigenvalues of $Q$ on the diagonal and $P$ an orthogonal matrix which contains the corresponding eigenvectors. Hence, we have
\begin{align}
    \left\langle y-x,\left(Q+aId\right)\left(y-x\right)\right\rangle &= \left\langle y-x,P \left(D+aId\right)P^\top \left(y-x\right)\right\rangle \nonumber\\
    & = \left\langle P^\top (y-x),\left(D+a Id\right)P^\top \left(y-x\right)\right\rangle. \label{eq:example_1 1.2}
\end{align}
Plugging \eqref{eq:example_1 1.2} back into \eqref{eq:example_1 1}, we obtain
\begin{align}
    \left\langle P^\top (y-x),\left(D+a Id\right)P^\top \left(y-x\right)\right\rangle 
    & \geq\left\langle u-2\left(Q+aId\right) x, y-x\right\rangle .\label{eq:example_1 2}
\end{align}
As $P P^\top = Id$ is the identity matrix, by changing the variables $\overline{y} = P^\top y, \overline{x} = P^\top x$, \eqref{eq:example_1 2} turns into
\begin{equation}
    \left\langle \overline{y} - \overline{x},\left(D+a Id\right) \left(\overline{y}-\overline{x}\right)\right\rangle 
    \geq\left\langle P^\top u-2 \left(D+a Id\right) \overline{x}, \overline{y}-\overline{x}\right\rangle .\label{eq:example_1 3}
\end{equation}
Because \eqref{eq:example_1 2} holds for all $y\in \R^n$ and $P$ is orthogonal, so \eqref{eq:example_1 3} has to hold for all $\overline{y}\in \R^n$. 

As we are in $\mathbb{R}^n$, let us write \eqref{eq:example_1 3} explicitly 
\begin{equation}
    \label{eq:ex2 sum 1}
    \sum_{i=1}^n (d_i+a)(\overline{y}_i - \overline{x}_i)^2 \geq \sum_{i=1}^n (P^{\top}_i u -2(d_i+a) \overline{x}_i) (\overline{y}_i - \overline{x}_i),
\end{equation}
where $d_i$ is the $i$-th eigenvalue of $Q$ and $P^{\top}_i$ is the $i$-th ow of matrix $P^{\top}$.

\begin{itemize}
\item If there is $j\in\N$ such that $d_j+a = 0$ then we have, from \eqref{eq:ex2 sum 1}
\begin{align*}
    &\sum_{i\neq j =1}^n (d_i+a)\left[ (\overline{y}_i - \overline{x}_i) - \frac{P^{\top}_i u -2(d_i+a) \overline{x}_i}{2(d_i +a)}\right]^2 \nonumber\\
    &\geq \sum_{i\neq j=1}^n \frac{(P^{\top}_i u -2(d_i+a)  \overline{x}_i)^2}{4(d_i +a)}  +(P^{\top}_j u )(\overline{y}_j - \overline{x}_j).
\end{align*}
Note that all components of $\overline{y}$ are separable. This holds for all $\overline{y}\in\R^n$ so we can choose $\overline{y}_i=\overline{x}_i$ 
\begin{equation*}
    (\forall \overline{y}_j\in\R) \qquad 0 \geq (P^{\top}_j u )(\overline{y}_j - \overline{x}_j),
\end{equation*}
which implies $P^{\top}_j u = 0$. We can extended this to the case where there are repeated eigenvalues i.e. $d_j +a= d_{j+1}+a=\dots =d_k +a=0$ for $1\leq j <k \leq n$.
\item Now we assume that $d_i+a \neq 0$ for all $i=1,\dots,n$, we have 
\begin{equation}
    \label{eq: example 2 sum square}
    \sum_{i=1}^n (d_i+a)\left[ (\overline{y}_i - \overline{x}_i) - \frac{P^{\top}_i u -2(d_i+a) \overline{x}_i}{2(d_i +a)}\right]^2  \geq \sum_{i=1}^n \frac{(P^{\top}_i u -2(d_i+a) \overline{x}_i)^2}{4(d_i +a)} ,
\end{equation}
If there is $1\leq j\leq n$ such that $d_j +a <0$ while $i\neq j \leq n, d_i+a >0$
\begin{align}
    &\sum_{i\neq j =1}^n (d_i+a)\left[ (\overline{y}_i - \overline{x}_i) - \frac{P^{\top}_i u -2(d_i+a) \overline{x}_i}{2(d_i +a)}\right]^2  -\frac{(P^{\top}_j u -2(d_j+a)  \overline{x}_j)^2}{4(d_j +a)} \nonumber\\
    &\geq \sum_{i\neq j=1}^n \frac{(P^{\top}_i u -2(d_i+a)  \overline{x}_i)^2}{4(d_i +a)}  -(d_j+a)\left[ (\overline{y}_j - \overline{x}_j) - \frac{P^{\top}_j u -2(d_j+a) \overline{x}_j}{2(d_j +a)}\right]^2.
\end{align}
Both sides are non-negative and as $\overline{y}$ is separable, we follow the same argument as in the previous case to obtain
\begin{equation*}
    (\forall \overline{y}_j\in\R)\qquad -\frac{(P^{\top}_j u -2(d_j+a)  \overline{x}_j)^2}{4(d_j +a)} \geq -(d_j+a)\left[ (\overline{y}_j - \overline{x}_j) - \frac{P^{\top}_j u -2(d_j+a) \overline{x}_j}{2(d_j +a)}\right]^2,
\end{equation*}
This implies $d_j+a=0$ or the above inequality cannot hold. Therefore, we only need to consider $d_i+a >0$ for all $1\leq i\leq n$. From \eqref{eq: example 2 sum square}, we derive that RHS is zero which is $P^{\top}_i u -2(d_i+a) \overline{x}_i = 0$ for all $1\leq i\leq n$. 
\end{itemize}
Combining all the cases above, we obtain that $\left(D+a Id\right)$ is non-negative semi-definite and $P^{\top} u- 2\left(D+a Id\right)\overline{x} =0$.
Hence, $a$ has to take a value $a\geq -\min \lambda_Q $ where $\lambda_{Q}$ are the eigenvalues of $Q$.
In conclusion, 
\[
\partial_{lsc}^\R f\left(x\right)=\left\{ \phi\in\Phi_{lsc}^\R :a\geq -\min \lambda_Q ,u= 2\left(Q+a  Id\right) x\right\} .
\]

\end{example}

\begin{example}
Let us consider the indicator function of a nonempty set $C$ of Hilbert space $X$
\[
f\left(x\right)=\iota_{C}\left(x\right)=\begin{cases}
0 & x\in C\\
+\infty & \text{otherwise}
\end{cases}
\]
and calculate $\lsc^\R$-subgradients of $f$ at $x\in X$. If $x\notin C$ then $\partial_{lsc}^\R f\left(x\right)=\emptyset$, so we consider $x\in C$. Let us take $(a,u) \in\partial_{lsc}^\R f(x)$
\begin{align*}
(\forall y\in X) \qquad 
f\left(y\right) & \geq-a\left(\left\Vert y\right\Vert ^{2}-\left\Vert x\right\Vert ^{2}\right)+\left\langle u,y-x\right\rangle .
\end{align*}
When $y\notin C$, the LHS becomes $+\infty$
and the inequality if fulfilled for all $a$ and $u$. Hence, the above inequality reduces to
\begin{equation}
(\forall y\in C) \quad a\left\Vert y\right\Vert ^{2}-\left\langle u,y\right\rangle \geq a\left\Vert x\right\Vert ^{2}-\left\langle u,x\right\rangle .\label{eq:subgrad indicator a not 0}
\end{equation}
\begin{itemize} 
\item If $a=0$ then we have $\langle u,y-x\rangle \leq 0$ for all $y\in C$, this means $u$ belongs to the normal cone of $C$ at $x$ in the sense of convex analysis. 
\item If $a>0$, we complete the square in \eqref{eq:subgrad indicator a not 0} to obtain
\begin{equation}
a\left\Vert y-\frac{u}{2a}\right\Vert ^{2}\geq a\left\Vert x-\frac{u}{2a}\right\Vert ^{2}.\label{eq:subgrad ind a not 0-1}
\end{equation}
Since $a>0$ and this holds for all $y\in C$, we can remove $a$
and taking the infimum with respect to $y\in C$ 
\[
\inf_{y\in C}\left\Vert y-\frac{u}{2a}\right\Vert ^{2}\geq\left\Vert x-\frac{u}{2a}\right\Vert ^{2}.
\]
This means that $x$ is a projection of $u/(2a)$ onto the set $C$.

\item If $a < 0$, then \eqref{eq:subgrad ind a not 0-1} changes sign when dividing by $a$
\begin{equation}
    \left\Vert y-\frac{u}{2a}\right\Vert ^{2}\leq \left\Vert x-\frac{u}{2a}\right\Vert ^{2},
\end{equation}
This leads to $x \in \argmax_{y \in C} \left\Vert y-\frac{u}{2a}\right\Vert$.
\end{itemize}
In conclusion, for $x\in C$
\begin{equation}
    \partial_{lsc}^\R \iota_C \left(x\right) \subseteq \left\{ \phi=(a,u) \in\lsc^\R: 
    \begin{cases}
    (\forall y\in C)\quad \langle u,y-x\rangle \leq 0 & a=0 \\
    x \in \argmin_{y\in C} \Vert y - \frac{u}{2a}\Vert & a >0 \\
    x \in \argmax_{y\in C} \Vert y - \frac{u}{2a}\Vert & a<0 
    \end{cases}
    \right\}.
    \label{eq: subgrad indicator C final}
\end{equation}
In fact if $(a,u)$ belongs to the set on the RHS of \eqref{eq: subgrad indicator C final}, then $(a,u)\in\partial_{lsc}^\R  f(x)$ by the same calculation. Finally,
\begin{equation*}
    \partial_{lsc}^\R \iota_C \left(x\right) = \left\{ \phi=(a,u) \in\lsc^\R: 
    \begin{cases}
    (\forall y\in C)\quad \langle u,y-x\rangle \leq 0 & a=0 \\
    x \in \argmin_{y\in C} \Vert y - \frac{u}{2a}\Vert & a >0 \\
    x \in \argmax_{y\in C} \Vert y - \frac{u}{2a}\Vert & a<0 
    \end{cases}
    \right\}.
\end{equation*}
\end{example}
Observe that if $a=0$ then $\lsc^0$-convex $f$ implies $f$ is convex, i.e. $\iota_C$ is convex or $C$ is convex.
In fact, we show that there is a connection between $\partial_{lsc}^\R \iota_C (x)$ and the proximal normal cone \cite[Chapter 1.1]{clarke2008nonsmooth} which is defined as follows: For a point $x\in C$, the proximal normal cone to the set $C$ at $x$ is the set
\begin{equation}
    \label{eq: proximal normal cone}
    N_p (x,C) =\left\{ v\in X: \exists t>0, \mathrm{ dist} (x+tv,C) = t\Vert v\Vert\right\}.
\end{equation}
\begin{proposition}
For every nonempty set $C\subset X$ and $x\in C$, if $\phi=(a,u)\in \partial_{lsc}^\R \iota_C (x)$ then 
\[
\nabla \phi (x) = u-2a x \in N_p (x,C).
\]
On the other hand, for $v\in N_p (x,C)$ there exists $\phi\in\partial_{lsc}^\R \iota_C (x)$ such that $\nabla \phi (x) = v$.     
\end{proposition}
\begin{proof}
    Let $x\in C$ and $\phi=(a,u) \in \partial_{lsc}^\R \iota_C (x)$. 
If $a\leq 0$, from \eqref{eq:subgrad indicator a not 0} we have for all $y\in C$
\begin{equation}
a\left\Vert y\right\Vert ^{2}-\left\langle u,y\right\rangle \geq a\left\Vert x\right\Vert ^{2}-\left\langle u,x\right\rangle \Rightarrow 0\geq  a \Vert y-x\Vert^2 \geq \langle u-2a x,y-x\rangle.
\end{equation}
The RHS can be further expressed
\begin{equation}
    0\geq \langle u-2a x,y-x\rangle = \frac{1}{2}\left[ \Vert u-2ax\Vert^2 + \Vert y-x\Vert^2 - \Vert y-x -(u-2ax)\Vert^2 \right]
\end{equation}
and finally
\begin{equation}
    \Vert y-x-(u-2ax)\Vert^2 \geq \Vert u-2ax\Vert^2.
\end{equation}
This holds for all $y\in C$ which implies that $x$ is a projection of $(x+u-2ax)$ onto $C$ which means $u-2ax \in N_p(x,C)$ with $t=1$.\\
If $a>0$, then 
\begin{equation*}
    0 \geq \langle \frac{u-2ax}{a},y-x\rangle - \Vert y-x\Vert^2 = \frac{1}{2}\left( \left\Vert \frac{u-2ax}{2a}\right\Vert^2 - \left\Vert y-x -\frac{u-2ax}{2a} \right\Vert^2\right),
\end{equation*}
or 
\begin{equation*}
    \left\Vert y-x -\frac{u-2ax}{2a} \right\Vert^2 \geq \left\Vert \frac{u-2ax}{2a}\right\Vert^2 \quad (\forall y\in C).
\end{equation*}
Thus, $x$ is the projection of $x+\frac{u-2ax}{2a}$ onto $C$ which means $u-2ax\in N_p (x,C)$ with $t=\frac{1}{2a}$. In both cases, $\nabla \phi (x) = u-2ax \in N_p (x,C)$.\\
Conversely, let $v\in N_p (x,C)$. There exists $t>0$ such that $x\in \Proj_C (x+tv)$ or
\begin{equation*}
    (\forall y\in C)\quad \Vert tv\Vert^2 \leq \Vert y-x - tv\Vert^2,
\end{equation*}
which can be expressed as
\begin{equation*}
    -\frac{1}{2t}\Vert y\Vert^2 +\langle v+\frac{x}{t}, y\rangle \leq -\frac{1}{2t}\Vert x\Vert^2 +\langle v+\frac{x}{t}, x\rangle,
\end{equation*}
which implies that $\phi(\cdot) = -\frac{1}{2t} \Vert \cdot \Vert^2 +\langle v+\frac{x}{t},\cdot\rangle\in \partial_{lsc}^\R  \iota_C (x)$ and $\nabla \phi (x) = v$.
\end{proof}

From the definition of $\lsc^\R$-subdifferentials, it is obvious that $\phi=\left(0,0\right)\in\partial_{lsc}^\R f\left(x_{0}\right)$, if and only if $x_{0}$ is a global minimizer of $f$ (c.f. \cite[Proposition 7.13]{Rub2013}).
In Example \ref{ex:1 J function}, the only case where $\left(0,0\right)\in\partial_{lsc}^\R  g_\gamma \left(x_{0}\right)$
is $x_{0}=0$. In the sequel, we use the concept of $a_{0}$-critical points as defined below.
\begin{definition}
\label{def: lsc stationary}
Let $f:X\to (-\infty,+\infty]$, a point $x_0\in X$ is $a_0$-critical point if $(a_0,2a_0 x_0) \in \partial_{lsc}^\R f(x_0)$, in other words
\begin{equation}
    \label{eq:lsc stationary}
    (\forall x\in X) \ f(x) -f(x_0) \geq -a_0 \Vert x-x_0\Vert^2
\end{equation}
\end{definition}
If $a_0 >0$, then $a_0$-criticality has been applied in \cite{davis2019,bednarczuk2023convergence}, to analyze algorithms involving the class of weakly convex functions.

If $a_0 \leq 0$, then \eqref{eq:lsc stationary} implies $x_0$ is the global minimizer of $f$.

\section{$\lsc^\R$-Proximal Operator}
\label{sec: prox operator}
Let $X$ be a Hilbert space. In this subsection, we introduce the proximal operator related to the class $\lsc^\R$-convex functions.
Observe that $\lsc^\R$-subdifferentials lie in the set $\lsc^\R$ which is not a subset of $X$. 
Therefore, we need a mapping which can serve as a link between $X$ and $\lsc^\R$. In analogy to the classical constructions we consider the function $g_\gamma$ from Example \ref{ex:1 J function}. 
\begin{definition}
\label{def: Jgamma duality map}
Let $\gamma>0$ and $g_\gamma(x) = \frac{1}{2\gamma}\Vert x\Vert^2$, we define \emph{$\lsc^\R$-duality map} $J_\gamma: X \rightrightarrows \lsc^\R$ as  
\begin{equation*}
    J_\gamma (x) := \partial_{lsc}^\R g_\gamma (x) =\partial_{lsc}^\R \left( \frac{1}{2\gamma}\Vert x \Vert^2 \right).
\end{equation*} 
Its inverse $J^{-1}_\gamma :\lsc^\R\rightrightarrows X$ is
\begin{equation*}
    J^{-1}_\gamma (\phi) = (\partial_{lsc}^\R g_\gamma)^{-1} (\phi).
\end{equation*} 
The explicit form of $J_\gamma$ and $J^{-1}_\gamma (\phi)$ are calculated in Example \ref{ex:1 J function}.
\end{definition}
When $\gamma=1$ then $J_1 =\partial_{lsc}^0 (\frac{1}{2}\Vert \cdot\Vert^2)$ in the sense of convex analysis subdifferentials, then we recover the classical duality mapping known in \cite[Example 2.26]{phelps2009convex}.

We consider the problem 
\begin{equation}
    \label{prob: min f}
    \min_{x\in X} f(x),
\end{equation}
where $f:X\to(-\infty,+\infty]$ which is proper $\lsc^\R$-convex.
Solving \eqref{prob: min f} means that we need to find $x_0 \in \dom f$ such that
\begin{equation}
\label{eq: optimality condition}
(0,0) \in\partial_{lsc}^\R f\left(x_{0}\right).
\end{equation}
By using Definition \ref{def: Jgamma duality map}, we define $\lsc^\R$-proximal operator $\prox^{\mathrm{lsc},\R}_{\gamma f}: X\rightrightarrows X$ (set-valued map)
\begin{equation}
\proxlsc_{\gamma f} (x) := \left( J_\gamma + \partial_{lsc}^\R f\right)^{-1} J_\gamma (x).
\label{eq:proximal-like operator 1}
\end{equation}
The concept of $\proxlsc_{\gamma f}$ is related to the concept of resolvent operator which is defined for classical convex subdifferentials as $(Id+\partial_{lsc}^0 f)^{-1}$. When $f$ is convex, this is also known as proximity operator
\begin{equation}
\label{eq: proximal operator}
\mathrm{prox}_{\gamma f}\left(x_{0}\right)=\arg\min_{z\in X}f\left(z\right)+\frac{1}{2\gamma}\left\Vert z-x_{0}\right\Vert ^{2}.
\end{equation}
\begin{remark}
    When reduced to convex analysis, the mapping $J_\gamma$ is actually the gradient of the norm square function which is single-valued and this implies the equivalence between optimality condition \eqref{eq: optimality condition} and the fixed point of the proximal operator \cite[Proposition 12.29]{Bau2011}. While in the weakly convex case, we have $\partial_{lsc}^a f(x)$ as the subdifferentials with fixed $a$. Then the fixed point of proximal operator is equivalent to $a$-critical point of $f$ \cite[Corollary 1]{bednarczuk2023calculus}.
\end{remark}

For the class of $\lsc^\R$-convex functions, we show that global minimizers of $f$ are related to fixed points of the $\lsc^\R$-proximal operator of $f$.

\begin{theorem}
\label{thm: fix point and optimal sol}
    Let $X$ be a Hilbert space. Let $f:X\to(-\infty,+\infty]$ be a proper $\lsc^\R$-convex function. If $x_0$ is a global minimizer of $f$ then $x_0$ is a fixed point of $\proxlsc_{\gamma f}$. 
    Conversely, if $x_0 \in \proxlsc_{\gamma f} (x_0)$ then $x_0$ is a critical point in the sense of Definition \ref{def: lsc stationary}. Additionally, if either 
    \begin{enumerate}
        \item $(-1/2\gamma,0) \in \partial_{lsc}^\R (f +\frac{1}{2\gamma}\Vert \cdot\Vert^2) (x_0)$, or 
        \item $\exists \phi_1,\phi_2 \in J_\gamma (x_0), \phi_1 -\phi_2 = (a_1-a_2,u_1-u_2) \in \partial_{lsc}^\R f(x_0), a_1 \leq a_2.$ 
    \end{enumerate}
Then $x_0$ is a global minimizer of $f$.
\end{theorem}
\begin{proof}
Let $x_0$ be a global minimizer of $f$. Then $x_0$ satisfies \eqref{eq: optimality condition}, as $(0,0) \in \partial_{lsc}^\R f(x_0)$, and
\begin{equation*}
    \phi \in \partial_{lsc}^\R f(x_0) +J_\gamma (x_0),
\end{equation*}
for any $\phi \in J_\gamma (x_0)$. This means that 
\begin{equation}
    \label{eq: fixed point of resolvent}
    x_0 \in \left( \partial_{lsc}^\R f +J_\gamma\right)^{-1} (\phi) \subset \left( \partial_{lsc}^\R f +J_\gamma\right)^{-1} J_\gamma (x_0).
\end{equation}    
Therefore, $x_0\in \proxlsc_{\gamma f} (x_0)$. 

On the other hand, let us assume that $x_0$ is a fixed point of $\proxlsc_{\gamma f}$. From \eqref{eq: fixed point of resolvent},
there exists $\phi_1\in J_\gamma (x_0)$ such that
\begin{equation}
    \label{eq: fixed point proximal to minimizer}
    \phi_1 \in \partial_{lsc}^\R f(x_0) +J_\gamma (x_0).
\end{equation}
If $\phi_1 = (-\frac{1}{2\gamma},0)$ then $x_0$ is a minimizer of $f$ as
\begin{equation*}
    (-\frac{1}{2\gamma},0) \in \partial_{lsc}^\R f(x_0) +J_\gamma (x_0) \subseteq \partial_{lsc}^\R (f +\frac{1}{2\gamma}\Vert \cdot\Vert^2) (x_0).
\end{equation*}
If $\phi_1 = (a_1, (\frac{1}{\gamma}+2a_1 )x_0), a_1 > - 1/2\gamma$, by \eqref{eq: fixed point proximal to minimizer}, there exist $\phi_2 \in J_\gamma (x_0)$ and $\phi_3 \in \partial_{lsc}^\R f(x_0)$ such that $\phi_1 = \phi_2 +\phi_3$ or
\begin{equation*}
    \phi_1 -\phi_2 =\phi_3 \in \partial_{lsc}^\R f(x_0)
\end{equation*}
As $\phi_2 \in J_\gamma (x_0), \phi_2 = (a_2, (\frac{1}{\gamma}+2a_2 )x_0)$ for $a_1, a_2$ such that $2\gamma a_2 \geq -1$.
By the definition of $\lsc^\R$-subdifferentials, for all $y\in X$
\begin{align}
    f(y) -f(x_0) &\geq (\phi_1 -\phi_2) (y) -(\phi_1 -\phi_2) (x_0) \nonumber\\
    &\geq -(a_1 -a_2) \left( \Vert y \Vert^2 -\Vert x_0 \Vert^2\right) +\langle  (\frac{1}{\gamma}+2a_1 )x_0 - (\frac{1}{\gamma}+2a_2 )x_0,y-x_0\rangle \nonumber\\
    & = -(a_1 -a_2) \left( \Vert y \Vert^2 -\Vert x_0 \Vert^2\right) +2(a_1 -a_2)\langle x_0, y-x_0\rangle \nonumber\\
    & = (a_2 - a_1) \Vert y-x_0 \Vert^2 \label{eq: fy-fx norm},
\end{align}
the right hand side is non-negative if $a_2\geq a_1$, which implies that $x_0$ is a minimizer of $f$. From \eqref{eq: fy-fx norm}, in general, we have $x_0$ is $(a_2-a_1)$-critical point of $f$. 
\end{proof}

The next result characterizes elements of $\lsc^\R$-proximal map.
\begin{theorem}
\label{thm: lsc prox equiv prox}
Let $X$ be a Hilbert space and $f:X\to (-\infty,+\infty]$ be a proper $\lsc^\R$-convex function. Let $x_0 \in \dom f, \gamma>0$. Then  
\begin{align}
    \label{eq: lsc prox equiv prox}
    x \in \proxlsc_{\gamma f} (x_0) &\Leftrightarrow \exists a_0 \geq -\frac{1}{2\gamma} \text{ s.t. } x\in \arg\min_{z \in X} \left[ f(z) +\left( \frac{1}{2\gamma} +a_0 \right) \Vert z-x_0\Vert^2 \right] 
\end{align}
\end{theorem}
\begin{proof}
Let $x \in \proxlsc_{\gamma f} (x_0)$. By definition, there exists $\phi_0\in J_\gamma (x_0)$ such that
\[
x\in (J_\gamma+\partial_{lsc}^\R f)^{-1} \phi_0.
\]
By \eqref{eq: J_gamma form}, there exists $a_0 \geq -1/2\gamma$ such that 
$ \phi_0 :=\left(a_0,(\frac{1}{\gamma}+2a_0) x_0 \right) \in J_\gamma (x_0)$ which satisfies
\begin{equation*}
    \left(a_0,\left(\frac{1}{\gamma}+2a_0 \right) x_0 \right) \in \partial_{lsc}^\R f(x) +J_\gamma (x) \subseteq \partial_{lsc}^\R (f+ \frac{1}{2\gamma} \Vert \cdot\Vert^2 ) (x).
\end{equation*}
Hence, by the definition of $\lsc^\R$-subdifferential, for all $y\in X$
\begin{align}
    f(y) +\frac{1}{2\gamma}\Vert y\Vert^2 - f(x) -\frac{1}{2\gamma}\Vert x\Vert^2 
    & \geq -a_0 \left( \Vert y\Vert^2 - \Vert x\Vert^2 \right) + \left( \frac{1}{\gamma}+2 a_0 \right)\langle x_0 , y-x\rangle. \label{eq:thm lsc prox equiv prox-1}
\end{align}
After simplifying all the quadratic terms in \eqref{eq:thm lsc prox equiv prox-1}, we obtain
\begin{align}
    f(y) - f(x) & \geq -\left(\frac{1}{2\gamma} +a_0\right) \Vert y-x\Vert^2 - \left( \frac{1}{\gamma}+2 a_0 \right) \langle x,y-x\rangle + \left( \frac{1}{\gamma}+2 a_0 \right)\langle x_0 , y-x\rangle \nonumber\\
    & \geq -\left(\frac{1}{2\gamma} +a_0\right) \Vert y-x\Vert^2 + \left( \frac{1}{\gamma}+2 a_0 \right)\langle x_0 -x , y-x\rangle.
\end{align}
By writing the inner product in terms of the norms, we have
\begin{equation*}
    f(y) - f(x) \geq \left(\frac{1}{2\gamma} +a_0\right) \left[ \Vert x_0-x\Vert^2 -\Vert y-x_0\Vert^2 \right],
\end{equation*}
which results in
\begin{equation}
    \label{eq: thm proof prox a_0}
    x\in \arg\min_{y\in X} f(y)+ \left(\frac{1}{2\gamma}+a_0\right) \Vert y-x_0\Vert^2.
\end{equation}
Conversely, let us assume \eqref{eq: thm proof prox a_0}. Then for all $y\in X$
\begin{align}
     f(y) - f(x) &\geq \left(\frac{1}{2\gamma} +a_0\right) \left[ \Vert x_0-x\Vert^2 -\Vert y-x_0\Vert^2 \right] \nonumber\\
     & = - \left(\frac{1}{2\gamma} +a_0\right) \left( \Vert y\Vert^2 -\Vert x\Vert^2 \right) + \left(\frac{1}{\gamma} +2 a_0\right)\langle x_0, y-x \rangle.
\end{align}
This is equivalent to
\begin{equation*}
    \left(a_0, (\frac{1}{\gamma}+2a_0) x_0\right) \in \partial_{lsc}^\R \left( f+ \frac{1}{2\gamma}\Vert\cdot\Vert^2\right) (x).
\end{equation*}
To finish the proof, we need to prove
\begin{equation*}
    \partial_{lsc}^\R \left( f+ \frac{1}{2\gamma}\Vert\cdot\Vert^2\right) (x) = (\partial_{lsc}^\R f + J_\gamma)(x).
\end{equation*}
We only have to prove that 
\begin{equation*}
    \partial_{lsc}^\R \left( f+ \frac{1}{2\gamma}\Vert\cdot\Vert^2\right) (x) \subset (\partial_{lsc}^\R f + J_\gamma)(x).
\end{equation*}
Let us take $(a,u)\in\partial_{lsc}^\R \left( f+ \frac{1}{2\gamma}\Vert\cdot\Vert^2\right) (x)$, so $(a+\frac{1}{2\gamma},u)\in\partial_{lsc}^\R f(x)$ and $(-\frac{1}{2\gamma},0) \in J_\gamma (x)$ from Example \ref{ex:1 J function}. Hence, we conclude the proof.
\end{proof}
Assertion \eqref{eq: lsc prox equiv prox} shows that for every $\gamma>0$, the proximal operator $\prox_{\gamma f}$ from \eqref{eq: proximal operator} is an element of $\lsc^\R$-proximal operator of $f$.

\begin{example}
\label{ex:strongly convex function}
    We consider the function $f:\R\to \R, f(x) =\left|x\right|+x^{2}$
which is $2$-strongly convex. We calculate the elements of $\prox_{\gamma f}^{\mathrm{lsc},\R} (x)$ for any $x\in\R$. Firstly, we compute $\partial_{lsc}^\R f(x)$. 
\begin{itemize}
\item When $x=0$, we need to find $(a,u)\in\lsc^\R$ such that 
\begin{equation}
\label{eq:ex strconv x=0}
\left(\forall y\in\mathbb{R}\right)\qquad\left|y\right|+y^{2}\geq-ay^{2}+uy.
\end{equation}
We divide into three cases of $y\in\R$. When $y=0$, \eqref{eq:ex strconv x=0} holds for all $(a,u)\in\lsc^\R$.
When $y>0$, we have 
\[
\left(a+1\right)y+1-u\geq0
\]
for $a\geq-1$ and $u\leq1$.
When $y<0$, we have 
\[
\left(a+1\right)y-\left(1+u\right)\leq0,
\]
for $a\geq-1$ and $u\geq-1.$
Hence, $\partial_{lsc}^\R f\left(0\right)=\left\{ \left(a,u\right):a\geq-1,-1\leq u\leq1\right\} .$
\item When $x>0$, we have 
\begin{equation}
\label{eq:ex strconv x>0}
\left(\forall y\in\mathbb{R}\right)\qquad\left|y\right|+y^{2}-x-x^{2}\geq-a\left(y^{2}-x^{2}\right)+u\left(y-x\right).
\end{equation}
When $y\geq 0$, we simplify \eqref{eq:ex strconv x>0} 
\[
\left(a+1\right)\left(y^{2}-x^{2}\right)\geq\left(u-1\right)\left(y-x\right).
\]
If $a=-1$ then $u=1$, otherwise, we have 
\[
\left(a+1\right)\left[y-x-\frac{u-1-2\left(a+1\right)x}{2\left(a+1\right)}\right]^{2}\geq\frac{\left[u-1-2\left(a+1\right)x\right]^{2}}{4\left(a+1\right)},
\]
which implies $a>-1$ and $u=1+2\left(a+1\right)x$.
When $y<0$, we have 
\[
\left(a+1\right)\left(y^{2}-x^{2}\right)\geq\left(u+1\right)\left(y-x\right).
\]
Since $y-x\neq 0$, we can further have
\[
\left(a+1\right)\left(y+x\right)\leq\left(u+1\right).
\]
As $y<0$, we need $a\geq-1$ and $u\geq-1+\left(a+1\right)x$.
Hence, for $x>0$,
\[ \partial_{lsc}^\R f\left(x\right)=\left\{ \left(a,u\right):a\geq-1,u=1+2\left(a+1\right)x\right\} .\] 
\item When $x<0$, similar to the previous case, we obtain 
\[ \partial_{lsc}^\R f\left(x\right)=\left\{ \left(a,u\right):a\geq-1,u=-1+2\left(a+1\right)x\right\} .\]
\end{itemize}
In conclusion, 
\begin{equation}
\label{eq:ex strconv subgrad}
\partial_{lsc}^\R f\left(x\right)=\left\{ \left(a,u\right)\in\Phi_{lsc}^\R:\begin{cases}
a\geq-1,u=-1+2\left(a+1\right)x & \text{when }x<0\\
a\geq-1,u=1+2\left(a+1\right)x & \text{when }x>0\\
a\geq-1,-1\leq u\leq1 & \text{when }x=0
\end{cases}\right\} .
\end{equation}
Let us take $x_0\in X$, for a given $\left(a_{0},u_{0}\right)\in J_{\gamma}\left(x_{0}\right)$, $x \in\mathrm{prox}_{\gamma f}^{\mathrm{lsc},\R}\left(x_{0}\right)$
means there exists ${\left(a,u \right)\in J_{\gamma}\left(x \right)}$
such that 
\[
\left(a_{0}-a ,u_{0}-u\right)\in\partial_{lsc}^\R f\left(x \right).
\]
By \eqref{eq:ex strconv subgrad}, we consider the following cases
\begin{itemize}
\item If $x <0$ then $a $ has to satisfy $a_0-a\geq -1$ and
\begin{align*}
u_{0}-u & =-1+2\left(a_{0}-a +1\right)x .
\end{align*}
We substitute $u_{0}$ and $u$ from Example \ref{ex:1 J function}, and arrive at 
\begin{align*}
x & =\frac{\left(\frac{1}{\gamma}+2a_{0}\right)x_{0}+1}{\left(\frac{1}{\gamma}+2a_{0}+1\right)}.
\end{align*}
\item If $x >0$, we also have $a$ satisfies $a_{0}-a \geq-1$ and 
\[
x =\frac{\left(\frac{1}{\gamma}+2a_{0}\right)x_{0}-1}{\left(\frac{1}{\gamma}+2a_{0}+1\right)}.
\]
\item If $x =0$ then $1\leq u_0 -u \leq 1$ or
\[
-\frac{\gamma}{1+2\gamma a_0} \leq x_0\leq \frac{\gamma}{1+2\gamma a_0}.
\]
Notice that when $2\gamma a_0 =-1$ then $u_0 = 0$ which also satisfies the above inequality.
\end{itemize}
Hence, we have 
\[
\mathrm{prox}_{\gamma f}^{\mathrm{lsc},\R}\left(x_{0}\right)=\begin{cases}
\frac{\left(\frac{1}{\gamma}+2a_{0}\right)x_{0}+1}{\left(\frac{1}{\gamma}+2a_{0}+1\right)} & \text{when } \left( \frac{1}{\gamma}+2a_0\right) x_{0}<-1,\ 2\gamma a_{0}\geq -1\\
\frac{\left(\frac{1}{\gamma}+2a_{0}\right)x_{0}-1}{\left(\frac{1}{\gamma}+2a_{0}+1\right)} & \text{when } \left( \frac{1}{\gamma}+2a_0\right) x_{0}>1,\ 2\gamma a_{0}\geq -1\\
0 & \text{otherwise}
\end{cases}.
\]
\end{example}

\section{$\lsc^\R$-Proximal Point Algorithm}
\label{sec: lsc PPA}
\subsection{Auxiliary Convergence Results}
This subsection presents the auxiliary convergence results which will be used in convergence proofs of $\lsc^\R$-proximal point algorithm introduced in Section \ref{sec: lsc PPA} and $\lsc^\R$-forward-backward algorithm introduced in Section \ref{sec: FB algorithm}. These auxiliary results are based on Fej\'er and quasi-Fej\'er monotonicity of the iterate (cf. \cite[Chapter 5]{Bau2011}).
\begin{lemma}[Lemma 3.1, \cite{combettes2001quasi}]
\label{lem: monotone sequence notnorm}
Let $\chi \in (0,1]$, $(\alpha_n)_{n\in\N},(\beta_n)_{n\in\N},(\varepsilon_n)_{n\in\N}$ be non-negative sequences with ${\sum_{n\in\N} \varepsilon_n <+\infty}$ such that 
\begin{equation}
    \alpha_{n+1} \leq \chi \alpha_n -\beta_n +\varepsilon_n.
\end{equation}
Then
\begin{enumerate}[label=(\roman*)]
    \item $(\alpha_n)_{n\in\N}$ is bounded and converges.
    \item $(\beta_n)_{n\in\N}$ is summable.
    \item If $\chi< 1$ then $(\alpha_n)_{n\in\N}$ is summable.
\end{enumerate}
\end{lemma}

\begin{theorem}
\label{thm: general convergence results}
Let $h:X\to (-\infty,+\infty]$ be a proper $\lsc^\R$-convex function. Let $\left( x_n\right)_{n\in\N}$ be a sequence in $\dom h$, $(\alpha_n)_{n\in\N},(\beta_n)_{n\in\N}$ be positive sequences on the real line. Assuming the following holds for all $n\in\N$
\begin{equation}
    \label{eq: general sequence xn}
    \alpha_{n+1} \Vert x^* -x_{n+1}\Vert^2 \leq \alpha_n \Vert x^*-x_n\Vert^2 - \beta_n \Vert x_n -x_{n+1} \Vert^2,
\end{equation}
for some $x^*\in S= \argmin \ h \neq \emptyset$. The following holds
\begin{enumerate}[label=(\roman*)]
    \item $\left(\alpha_{n} \left\Vert x^{*}-x_{n}\right\Vert ^{2}\right)_{n\in\N}$
    converges.
    \item $\left( \alpha_n \mathrm{dist}^{2}\left(x_{n},S\right) \right)_{n\in\N}$ is decreasing and converges.
    \item $\sum_{n\in\N} \beta_n \Vert x_n -x_{n+1}\Vert^2 <+\infty$. If $0< \beta \leq \beta_n$ for all $n\in\N$ then $\sum_{n\in\N} \Vert x_{n} -x_{n+1}\Vert^2 <+\infty$.
    \item If $\sum_{n\in\N} 1/\sqrt{\beta_n} <+\infty$ then $\sum_{n\in\N} \Vert x_n-x_{n+1}\Vert <+\infty $.
\end{enumerate}
If $\alpha_n$ is non-decreasing then 
\begin{enumerate}[resume,label=(\roman*)]
    \item $\left(\left\Vert x^{*}-x_{n}\right\Vert ^{2} \right)_{n\in\N}$
    converges.
    \item $\left(\mathrm{dist}^{2}\left(x_{n},S\right) \right)_{n\in\N}$ is decreasing and converges.
\end{enumerate}
\end{theorem}

\begin{proof}
We can consider $\alpha_n \Vert x^*-x_n\Vert^2$ as $\tilde{\alpha}_n$ and $\beta_n \Vert x_n -x_{n+1} \Vert^2$ as $\tilde{\beta}_n$ and apply Lemma \ref{lem: monotone sequence notnorm} with $\varepsilon_n =0 , \chi =1$ to arrive at (i) and by Lemma \ref{lem: monotone sequence notnorm}-(ii)
\begin{equation*}
    \sum_{n\in\N} \beta_n \Vert x_n -x_{n+1}\Vert^2 <+\infty.
\end{equation*}
It follows from \eqref{eq: general sequence xn} for all $n\in\N$, we have
\begin{equation}
\label{eq: general Fejer sequence}
    \alpha_{n+1} \Vert x^* -x_{n+1}\Vert^2 \leq \alpha_n \Vert x^*-x_n\Vert^2.
\end{equation}
As $\alpha_n$ does not depend on $x^* \in S$ for all $n\in\N$, (ii) holds by taking the infimum with respect to $x^* \in S$ on both sides of \eqref{eq: general Fejer sequence}.

From \eqref{eq: general sequence xn}, we know that
\begin{equation*}
    \alpha_{n+1}\Vert x^* -x_{n+1}\Vert^2  +\beta \Vert x_n -x_{n+1} \Vert^2 
    \leq \alpha_{n+1}\Vert x^* -x_{n+1}\Vert^2  +\beta_n \Vert x_n -x_{n+1} \Vert^2 \leq  \alpha_n\Vert x^*-x_n\Vert^2,
\end{equation*}
which is
\begin{equation*}
     \beta \Vert x_n -x_{n+1} \Vert^2 
    \leq  \alpha_n\Vert x^*-x_n\Vert^2 - \alpha_{n+1}\Vert x^* -x_{n+1}\Vert^2.
\end{equation*}
By summing the above inequality from $n=0 $ to $N\in\N$ we get
\begin{equation*}
     \beta \sum_{n=0}^N \Vert x_n -x_{n+1} \Vert^2 
    \leq  \alpha_0\Vert x^*-x_0\Vert^2 - \alpha_{N+1}\Vert x^* -x_{N+1}\Vert^2 \leq \alpha_0\Vert x^*-x_0\Vert^2,
\end{equation*}
and letting $N$ go to infinity, we have (iii).  
To show (iv), we infer from \eqref{eq: general sequence xn},
\begin{equation}
    \Vert x_n -x_{n+1} \Vert \leq \sqrt{\frac{\alpha_n}{\beta_n}} \Vert x^* -x_{n}\Vert \leq \sqrt{\frac{\alpha_{n-1}}{\beta_n}} \Vert x^* -x_{n-1}\Vert \leq \sqrt{\frac{\alpha_0}{\beta_n}} \Vert x^* -x_0\Vert.
\end{equation}
Thanks to the assumption that $\sum_{n\in\N} \frac{1}{\sqrt{\beta_n}} <+\infty$, $\Vert x_n -x_{n+1}\Vert$ is summable. 

Let us assume that $\alpha_n$ is non-decreasing, since $\alpha_n >0$, we can divide both sides of \eqref{eq: general Fejer sequence} by $\alpha_{n+1}$ to obtain
\begin{equation}
    \Vert x^* -x_{n+1}\Vert^2 \leq \frac{\alpha_n}{\alpha_{n+1}} \Vert x^*-x_n\Vert^2 \leq \Vert x^*-x_n\Vert^2.
\end{equation}
This proves (v) and also (vi). 


\end{proof}

\subsection{$\lsc^\R$-Proximal Point Algorithm}

Let us assume that the function $f:X\to (\infty,+\infty]$ is proper $\lsc^\R$-convex, 
for all $x\in \dom f$. Let us further assume $f$ has a global minimizer i.e. the set $S=\argmin_{x\in X} f(x)$ is non-empty. 
The $\lsc^\R$-proximal point algorithm \eqref{eq:prox alg}, starting with $x_0\in \dom f$ and stepsize $\gamma>0$, is as follows
\begin{align}
x_{n+1} & \in \proxlsc_{\gamma f} (x_n) =\left( J_\gamma +\partial_{lsc}^\R f \right)^{-1}J_\gamma \left(x_{n}\right).\label{eq:prox alg}\tag{$\lsc^\R$-PPA}
\end{align}
According to \eqref{eq:proximal-like operator 1}, the following conditions must be satisfied for the $\proxlsc_{\gamma f}$ to be well-defined,
\begin{align}
(\forall n\in \N) \ J_\gamma\left(x_{n}\right)\cap \mathrm{ran }\left(J_\gamma +\partial_{lsc}^\R f\right) \neq \emptyset. \label{eq: prox-lsc well-defined}
\end{align}
If for all $n\in\mathbb{N}$, there exist $\phi_n=(a,u_n) \in J_\gamma (x_n)$ and $\phi_{n+1} =(a,u_{n+1}) \in J_\gamma (x_{n+1})$ for $2\gamma a \geq -1$ such that
\begin{equation}
    \label{eq: same constant subgrad}
    \phi_n-\phi_{n+1} = (0,u_n-u_{n+1}) = \left(0,\left(\frac{1}{\gamma} +2a\right) x_n- \left(\frac{1}{\gamma} +2a \right)x_{n+1}\right) \in \partial_{lsc}^\R f(x_{n+1}),
\end{equation}
(where the second equality come from Example \ref{ex:1 J function})
then we have
\begin{equation*}
    \left(\frac{1}{\gamma}+2a\right) (x_n-x_{n+1}) \in \partial_{lsc}^0 f(x_{n+1}),
\end{equation*}
where $\partial_{lsc}^0 f$ is a subgradient of $f$ in the sense of convex analysis. This is proximal point update in convex analysis with $(1/\gamma+2a)$ as the stepsize.

Combining with Theorem \ref{thm: lsc prox equiv prox}, we can run \eqref{eq:prox alg} as in algorithm \ref{alg:PPA-example}.

\begin{algorithm}[H]\caption{$\lsc^\R$-Proximal Point Algorithm}\label{alg:PPA-example}
\begin{enumerate}
    \item \textbf{Initialize:} $\gamma >0$, $x_0 \in \dom f$ and $\left( a_0, (\frac{1}{\gamma}+2a_0)x_0\right) \in J_\gamma (x_0)$
    \item \textbf{For $n>0$, update}
    \begin{itemize}
        \item Pick $\left( a_n, (\frac{1}{\gamma}+2a_n)x_n\right) \in J_\gamma (x_n)$
        \item $x_{n+1} \in \argmind{z\in X}{ f(z) +\left( \frac{1}{2\gamma} +a_n\right) \Vert z-x_n\Vert^2}$
    \end{itemize}
\end{enumerate}
\end{algorithm}
\vspace{0.5cm}

We present the result related to monotonicity of the objective function generated by \eqref{eq:prox alg}.

\begin{proposition}
\label{prop:prox f norm}
Let $f:X\to (-\infty,+\infty]$ be a proper, $\lsc^\R$-convex function, and \eqref{eq: prox-lsc well-defined} hold. Let $\left( x_n\right)_{n\in\N}$ and $(a_n)_{n\in\N}$ be sequences generated by \eqref{eq:prox alg} with $\gamma>0$.
Then it holds for all $y\in X$ that
\begin{equation}
    f(y) -f(x_{n+1}) \geq \left(\frac{1}{2\gamma}+a_{n+1}\right)\left\Vert y-x_{n+1}\right\Vert ^{2}+\left(\frac{1}{2\gamma}+a_{n}\right)\left( \left\Vert x_{n+1}-x_{n}\right\Vert ^{2} - \left\Vert y-x_{n}\right\Vert ^{2}\right).\label{eq:prox alg f norm}
\end{equation}
In particular, $(f(x_n))_{n\in\N}$ is decreasing.
\end{proposition}

\begin{proof}
By \eqref{eq: prox-lsc well-defined}, the sequence $(x_n)_{n\in\N}$ is well-defined, i.e. there exist $\left(a_{n},\left(\frac{1}{\gamma}+2a_{n}\right)x_{n}\right)\in J_\gamma \left(x_{n}\right)$
and $\left(a_{n+1},\left(\frac{1}{\gamma}+2a_{n+1}\right)x_{n+1}\right)\in J_\gamma \left(x_{n+1}\right)$ such that
\begin{align}
\left(a_{n},\left(\frac{1}{\gamma}+2a_{n}\right)x_{n}\right)- \left(a_{n+1},\left(\frac{1}{\gamma}+2a_{n+1}\right)x_{n+1}\right) & \in\partial_{lsc}^\R f\left(x_{n+1}\right),\label{eq:prox alg subdiff}
\end{align}
with the conditions $2\gamma a_n \geq -1$ and $2\gamma a_{n+1} \geq -1$.
By definition of $\Phi_{lsc}^\R$-subgradient of $f$ at $x_{n+1}$, we have
\begin{align}
\left(\forall y\in X\right)\quad f\left(y\right)-f\left(x_{n+1}\right) & \geq-\left(a_{n}-a_{n+1}\right)\left(\left\Vert y\right\Vert ^{2}-\left\Vert x_{n+1}\right\Vert ^{2}\right) \nonumber\\
& +\left\langle \frac{1}{\gamma}\left(x_{n}-x_{n+1}\right)+2\left(a_{n}x_{n}-a_{n+1}x_{n+1}\right),y-x_{n+1}\right\rangle \nonumber \\
 & =\frac{1}{\gamma}\left\langle x_{n}-x_{n+1},y-x_{n+1}\right\rangle +2\left\langle a_{n}x_{n}-a_{n+1}x_{n+1},y-x_{n+1}\right\rangle  \nonumber\\ 
 &-\left(a_{n}-a_{n+1}\right)\left(\left\Vert y\right\Vert ^{2}-\left\Vert x_{n+1}\right\Vert ^{2}\right)\label{eq:prox alg estimation f}
\end{align}
The last two terms on the right hand side can be simplified to
\begin{align}
& 2\left\langle a_{n}x_{n}-a_{n+1}x_{n+1},y-x_{n+1}\right\rangle -\left(a_{n}-a_{n+1}\right)\left(\left\Vert y\right\Vert ^{2}-\left\Vert x_{n+1}\right\Vert ^{2}\right) \nonumber\\
& = 2a_{n}\left\langle x_{n}-x_{n+1},y-x_{n+1}\right\rangle +\left(a_{n+1}-a_{n}\right)\left\Vert y-x_{n+1}\right\Vert ^{2}. \label{eq: lem eq 1}
\end{align}
Plugging \eqref{eq: lem eq 1} back into \eqref{eq:prox alg estimation f} we obtain
\begin{align}
f(y)-f(x_{n+1}) & \geq
\left(\frac{1}{\gamma}+2a_{n}\right)\left\langle x_{n}-x_{n+1},y-x_{n+1}\right\rangle +\left(a_{n+1}-a_{n}\right)\left\Vert y-x_{n+1}\right\Vert ^{2}\nonumber \\
 & =\left(\frac{1}{2\gamma}+a_{n}+a_{n+1}-a_{n}\right)\left\Vert y-x_{n+1}\right\Vert ^{2}+\left(\frac{1}{2\gamma}+a_{n}\right)\left\Vert x_{n+1}-x_{n}\right\Vert ^{2} \nonumber\\ 
 &-\left(\frac{1}{2\gamma}+a_{n}\right)\left\Vert y-x_{n}\right\Vert ^{2}\nonumber \\
 & =\left(\frac{1}{2\gamma}+a_{n+1}\right)\left\Vert y-x_{n+1}\right\Vert ^{2}-\left(\frac{1}{2\gamma}+a_{n}\right)\left\Vert y-x_{n}\right\Vert ^{2} \nonumber\\
 &+\left(\frac{1}{2\gamma}+a_{n}\right)\left\Vert x_{n+1}-x_{n}\right\Vert ^{2},
\end{align}
which proves \eqref{eq:prox alg f norm}.
By taking $y=x_{n}$, we obtain
\[
f\left(x_{n}\right)-f\left(x_{n+1}\right)\geq\left(\frac{1}{\gamma}+a_{n+1}+a_{n}\right)\left\Vert x_{n}-x_{n+1}\right\Vert ^{2}\geq0,
\]
as $\frac{1}{\gamma}+a_{n+1}+a_{n} \geq 0$ from the conditions of $a_n$ and $a_{n+1}$.
Therefore, $(f\left(x_{n}\right))_{n\in\N}$ is decreasing.
\end{proof}
Proposition \ref{prop:prox f norm} provides a description of the behavior of the objective function at each iterate. In the following result, the role of coefficients $a_n$ in the convergence results is investigated. Depending on the subdifferentials of $f$, the behavior of $(a_n)_{n\in\N}$ can be divided into two cases. Below, we present the main convergence result of \eqref{eq:prox alg}.

\begin{theorem}
\label{thm:Fejer prox pt}
Let $f:X\to (-\infty,+\infty]$ be a proper $\lsc^\R$-convex function with the set $S=\argmin_{x\in X} f(x)$ nonempty. Assuming  \eqref{eq: prox-lsc well-defined} holds, $\left( x_n\right)_{n\in\N}, (a_n)_{n\in\N}$ are sequences generated by \eqref{eq:prox alg} with $\gamma>0$,
then the following hold.
\begin{enumerate}[label=(\roman*)]
    \item If there exists $n_0\in\N$ such that $\frac{1}{2\gamma}+a_{n_0} =0$, $x_{n_0+1}$ is a global minimizer of $f$. 
    \item If $\frac{1}{2\gamma}+a_n >0$ for all $n\in\N$, we have $\lim_{n\to\infty} f(x_n) = f(x^*)$, where $x^*\in S$. The assertions (i)-(iv) of Theorem \ref{thm: general convergence results} hold with $\alpha_n= \beta_n = \frac{1}{2\gamma}+a_n$, with $n\in\N$. Moreover, if $(a_n)_{n\in\N}$ is non-decreasing, then we also have (v)-(vi) of Theorem \ref{thm: general convergence results} i.e. 
    \begin{enumerate}
    \item[(v)] $\left(\left\Vert x^{*}-x_{n}\right\Vert ^{2} \right)_{n\in\N}$
    converges.
    \item[(vi)] $\left(\mathrm{dist}^{2}\left(x_{n},S\right) \right)_{n\in\N}$ is decreasing and converges.
\end{enumerate}
\end{enumerate}

\end{theorem}
\begin{proof}
Thanks to assumption \eqref{eq: prox-lsc well-defined}, $(x_n)_{n\in\N}$ is well-defined. 
In case \textit{(i)}, there exists an $n_0\in\N$ such that $\frac{1}{2\gamma}+a_{n_0} = 0$, then at iteration $n_0 +1$, by Proposition \ref{prop:prox f norm},
\begin{equation}
   (\forall y\in X) \qquad f(y) -f(x_{n_0+1}) \geq \left(\frac{1}{2\gamma}+a_{n_0+1}\right)\left\Vert y-x_{n_0+1}\right\Vert ^{2} \geq 0.
\end{equation}
This means $x_{n_0+1}$ is the global minimizer, so we can stop the algorithm after $n_0+1$ steps no matter if $(\frac{1}{2\gamma}+a_{n_0+1})$ equals to zero or not.

In case \textit{(ii)}, from inequality \eqref{eq:prox alg f norm} in Proposition \ref{prop:prox f norm}, by taking $y=x^{*}\in S$, we obtain
\begin{align}
0\geq f\left(x^{*}\right)-f\left(x_{n+1}\right) & \geq\left(\frac{1}{2\gamma}+a_{n+1}\right)\left\Vert x^{*}-x_{n+1}\right\Vert ^{2}-\left(\frac{1}{2\gamma}+a_{n}\right)\left\Vert x^{*}-x_{n}\right\Vert ^{2} \nonumber\\
& +\left(\frac{1}{2\gamma}+a_{n}\right)\left\Vert x_{n+1}-x_{n}\right\Vert ^{2}. \label{eq: thm proof norm square}
\end{align}
which coincides with Theorem \ref{thm: general convergence results} inequality \eqref{eq: general sequence xn} for $\alpha_n = \beta_n = \frac{1}{2\gamma} +a_n > 0$ for all $n\in\N$. Notice that the monotonicity of $\alpha_n$ depends on the monotonicity of $a_n$. When $(a_n)_{n\in\N}$ is non-decreasing, Theorem \ref{thm: general convergence results}-{(v,vi)} hold for the function $h=f$.

For the limit of the objective function, using the second inequality of \eqref{eq: thm proof norm square} and Theorem \ref{thm: general convergence results}-(i), we can skip the last term to arrive at
\begin{align}
f\left(x^{*}\right)-f\left(x_{n+1}\right) & \geq\left(\frac{1}{2\gamma}+a_{n+1}\right)\left\Vert x^{*}-x_{n+1}\right\Vert ^{2}-\left(\frac{1}{2\gamma}+a_{n}\right)\left\Vert x^{*}-x_{n}\right\Vert ^{2}. 
\end{align}
Taking the limit on both sides, by Theorem \ref{thm: general convergence results}-(i), the right hand side converges to zero while the left hand side is non-positive, which implies that 
\begin{equation*}
    \lim_{n\to\infty} f(x_{n+1}) = f(x^*).
\end{equation*} 
\end{proof}

\begin{remark}
The $\lsc^\R$-subdifferentials of $f$ influent the behavior of $(a_n)_{n\in\N}$ in \eqref{eq:prox alg}. If for any $(a,u)\in\partial_{lsc}^\R f (x)$, $a> 0$ for all $x\in \dom f$, then by \eqref{eq:prox alg subdiff}, at each iteration, we must have $a_n -a_{n+1} > 0$ so $(a_n)_{n\in\N}$ is decreasing and consequently $\alpha_n$ is decreasing.
Then the assertions (v,vi,vii) of Theorem \ref{thm:Fejer prox pt} will not be met. For example, the function $f(x)=-x^2$ has $\partial_{lsc}^\R  f(x) = \{(a,u)\in\lsc^\R: a\geq 1, u=2(a-1)x\}$ for all $x\in\R$.


\end{remark}
If $\alpha_n$ is bounded from below by a positive constant in \eqref{eq:prox alg}, then Theorem \ref{thm: general convergence results}-(iii) implies $\Vert x_n -x_{n+1}\Vert^2$ is summable and $(x_n)_{n\in\N}$ is bounded.

To see that $(x_n)_{n\in\N}$ is bounded, let us fix $x^*\in S$ and let $A\leq \alpha_n$ for all $n\in\N$ for some constant $A>0$. By contradiction, assume that $(x_n)_{n\in\N}$ is unbounded i.e. there exists a subsequence $(x_{n_k})_{k\in\N}$ such that
\begin{equation*}
    A k^2 \leq \alpha_{n_k} \Vert x_{n_k}-x^*\Vert^2, \forall k\in\N.
\end{equation*}
By Theorem \ref{thm: general convergence results}-(i), $(\alpha_n \Vert x_n-x^*\Vert^2)_{n\in\N}$ converges. Let $\delta = \lim_{n\to \infty} \alpha_n \Vert x_n -x^*\Vert^2$, then
\begin{equation*}
    |A k^2 -\delta| \leq |\alpha_{n_k} \Vert x_{n_k}-x^*\Vert^2-\delta |, \ \forall k\in\N.
\end{equation*}
The left hand side tends to infinity while the right hand side goes to zero which leads to a contradiction.

Since $x_{n}$ is bounded, there exists a weakly convergent subsequence of $(x_n)_{n\in\N}$.
However, we do not know if a weak limit point of $(x_n)_{n\in\N}$ lies in $S$. This is because we only know that the function $f$ is lsc, thanks to $\lsc^\R$-convexity which does not mean that $f$ is weak lsc as, in general, it is not convex.


\begin{remark}
    Theorem \ref{thm:Fejer prox pt}-(i) provides a stopping criterion for \eqref{eq:prox alg}. As stated in Theorem \ref{thm: fix point and optimal sol}, a minimizer is also a fixed point of the $\lsc^\R$-proximal operator. However, as the $\lsc^\R$-proximal operator is a set-valued operator, it can return a fixed point and not a solution (see Theorem \ref{thm: fix point and optimal sol}).
\end{remark}

\section{$\lsc^\R$-Forward Backward Algorithm}
\label{sec: FB algorithm}
Inspired by the construction of $\lsc^\R$-proximal operator in \eqref{eq:proximal-like operator 1}, now we go a step further by considering the problem
\begin{equation*}
    \min_{x\in X} f(x) +g(x),
\end{equation*}
which can be approached by finding a point $x_0\in X$ such that
\begin{equation}
\label{eq: 0 in f+g}
(0,0)\in\left(\partial_{lsc}^\R f+\partial_{lsc}^\R g\right)\left(x_{0}\right) \subseteq \partial_{lsc}^\R  (f+g)(x_0),
\end{equation}
where the functions $f,g:X\to (-\infty,+\infty]$ are proper $\lsc^\R$-convex and $\dom f\cap \dom g \neq \emptyset$. Moreover, let us assume that the set $S=\argmin_{x\in X} (f+g)(x) $ is non-empty. 

For any $x_0\in\dom f\cap \dom g$, we define the following update 
\begin{equation}
x_{n+1}\in\left(J_\gamma +\partial_{lsc}^\R f\right)^{-1}\left(J_\gamma -\partial_{lsc}^\R g\right)\left(x_{n}\right).\label{eq:forward-backward-operator 1}\tag{$\lsc^\R$-FB}
\end{equation}
To ensure that the iterate in \eqref{eq:forward-backward-operator 1} is well-defined, we assume the following condition 
\begin{align}
(\forall n\in \N) \ \left(J_\gamma-\partial_{lsc}^\R g\right) (x_n) \cap \mathrm{ran }\left(J_\gamma +\partial_{lsc}^\R f\right) \neq \emptyset.
\label{eq:well-defined FB}
\end{align}
We refer to algorithm \eqref{eq:forward-backward-operator 1} as $\lsc^\R$-Forward-Backward Algorithm. 

Similar to Theorem \ref{thm: lsc prox equiv prox}, one can interpret \eqref{eq:forward-backward-operator 1} as follow.
\begin{theorem}
\label{thm: lsc FB equiv FB}
Let $X$ be a Hilbert space and $f,g:X\to (-\infty,+\infty]$ be proper $\lsc^\R$-convex functions. Let $x_0 \in \dom f, \gamma>0$. If 
\begin{align}
    \label{eq: lsc FB equiv FB}
    x \in\left(J_\gamma +\partial_{lsc}^\R f\right)^{-1}\left(J_\gamma -\partial_{lsc}^\R g\right)\left(x_{0}\right),
\end{align}
then there exists $(a_0,(1/\gamma+2a_0)x_0) \in J_\gamma (x_0)$ and $(a_0^g,u_0^g)\in\partial_{lsc}^\R g(x_0)$ such that
\[
x \in \argmind{y\in X}{ f(y) +\langle u_0^g-2a_0^g x_0,y\rangle} +\left( \frac{1}{2\gamma} +a_0-a_0^g\right)\Vert y-x_0\Vert^2.
\]
\end{theorem}
\begin{proof}
Let \eqref{eq: lsc FB equiv FB} hold, it means there exist
\[
\left( a_0,(\frac{1}{\gamma}+2a_0 )x_0\right) \in J_\gamma (x_0),\ \left( a,(\frac{1}{\gamma}+2a )x\right) \in J_\gamma (x), \ (a_0^g,u_0^g) \in \partial_{lsc}^\R g(x_0),
\]
such that
\[
\left( a_0,(\frac{1}{\gamma}+2a_0 )x_0\right) - \left( a,(\frac{1}{\gamma}+2a )x\right) - (a_0^g,u_0^g) \in \partial_{lsc}^\R f(x).
\]
Using $\lsc^\R$-subdifferentials gives us, for all $y\in X$,
\begin{align*}
    f(y)-f(x) & \geq -\left( a_0 -a -a_0^g\right)(\Vert y\Vert^2 - \Vert x \Vert^2) +\langle (\frac{1}{\gamma}+2a_0 )x_0 - (\frac{1}{\gamma}+2a )x -u_0^g ,y-x \rangle\\
    & = -\left( a_0 -a -a_0^g\right)\Vert y-x \Vert^2 + \langle (\frac{1}{\gamma}+2a_0 )(x_0 - x ) -u_0^g +2a_0^g x ,y-x \rangle \\
    & = \left( \frac{1}{2\gamma}+a +a_0^g\right)\Vert y-x \Vert^2 - \left( \frac{1}{2\gamma}+ a_0 \right)\Vert y-x_{0}\Vert^2 +\left(\frac{1}{2\gamma} +a_0 \right)\Vert x_0 -x \Vert^2\\
    & + \langle 2a_0^g x_0 -u_0^g ,y-x \rangle +2a_0^g \langle  x- x_0,y-x \rangle \\
    & = \left( \frac{1}{2\gamma}+a \right)\Vert y-x \Vert^2 - \left( \frac{1}{2\gamma}+ a_0 -a_0^g\right)\Vert y-x_{0}\Vert^2 +\left(\frac{1}{2\gamma} +a_0 -a_0^g\right)\Vert x_0-x \Vert^2\\
    & -\langle u_0^g-2a_0^g x_0,y-x \rangle.
\end{align*}
As $\left( a,(\frac{1}{\gamma}+2a )x\right) \in J_\gamma (x)$, we have $2\gamma a\geq -1$ which infers
\begin{align*}
    f(y)-f(x) & \geq - \left( \frac{1}{2\gamma}+ a_0 -a_0^g\right)\Vert y-x_{0}\Vert^2 +\left(\frac{1}{2\gamma} +a_0 -a_0^g\right)\Vert x_0-x \Vert^2 -\langle u_0^g -2a_0^g x_0,y-x \rangle.
\end{align*}
or
\[
x \in \argmind{y\in X}{ f(y)+\langle u_0^g -2a_0^g x_0,y \rangle  +\left(\frac{1}{2\gamma} +a_0 -a_0^g\right)\Vert y-x_{0}\Vert^2}.
\]
\end{proof}

From Theorem \ref{thm: lsc FB equiv FB}, we propose algorithm \ref{alg:FB-example} which is one way to implement \eqref{eq:forward-backward-operator 1}.
\begin{algorithm}[H]\caption{$\lsc^\R$-Forward-Backward Algorithm}\label{alg:FB-example}
\begin{enumerate}
    \item \textbf{Initialize:} $\gamma >0$, $x_0 \in \dom f$ and $\left( a_0, (\frac{1}{\gamma}+2a_0)x_0\right) \in J_\gamma (x_0), (a_0^g,u_0^g) \in \partial_{lsc}^\R g(x_0)$
    \item \textbf{For $n>0$, update}
    \begin{itemize}
        \item Pick $\left( a_n, (\frac{1}{\gamma}+2a_n)x_n\right) \in J_\gamma (x_n), (a_n^g,u_n^g) \in \partial_{lsc}^\R g(x_n)$
        \item $x_{n+1} \in \argmind{z\in X}{ f(z) + \langle u_n^g -2a_n^g x_n,z\rangle +\left( \frac{1}{2\gamma} +a_n-a_n^g\right) \Vert z-x_n\Vert^2}$
    \end{itemize}
\end{enumerate}
\end{algorithm}
\vspace{0.5cm}

We start with the following technical fact

\begin{proposition}
    Let $g:X\to (-\infty,+\infty]$ be proper $\lsc^\R$-convex. If $\left(J_\gamma -\partial_{lsc}^\R g\right) \subseteq \dom J_\gamma^{-1}$, then 
\begin{equation}
    \left(J_\gamma +\partial_{lsc}^\R f\right)^{-1}\left(J_\gamma -\partial_{lsc}^\R g\right) \subset \left(J_\gamma +\partial_{lsc}^\R f\right)^{-1} J_\gamma J_\gamma^{-1}\left(J_\gamma -\partial_{lsc}^\R g\right) = \proxlsc_{\gamma f} (J_\gamma^{-1}\left(J_\gamma -\partial_{lsc}^\R g\right)).
    \label{eq: FB in prox term}
\end{equation}
\end{proposition}


\begin{proof}
Let $x\in \dom f\cap \dom g$ take $\phi\in \left(J_\gamma -\partial_{lsc}^\R g\right) (x)$. Then by assumption, $\phi \in \dom J_\gamma^{-1}$ which implies $\phi \in J_\gamma J_\gamma^{-1} (\phi)$ and
\[
\phi \in J_\gamma J_\gamma^{-1} \left(J_\gamma -\partial_{lsc}^\R g\right) (x).
\]
Then it is obvious that 
\[
\left(J_\gamma +\partial_{lsc}^\R f\right)^{-1}\left(J_\gamma -\partial_{lsc}^\R g\right) \subset \left(J_\gamma +\partial_{lsc}^\R f\right)^{-1} J_\gamma J_\gamma^{-1}\left(J_\gamma -\partial_{lsc}^\R g\right).
\]
\end{proof}

Let us consider the function $g$ to be Fr{\'e}chet differentiable on the whole domain with Lipschitz continuous gradient with Lipschitz constant $L_g>0$. 
We have an interesting relationship between $\lsc^\R$-subdifferentials and gradient of $g$.
\begin{lemma}
\label{lem: lsc subdiff and grad}
Let $g:X\to (-\infty,+\infty]$ be Fr{\'e}chet-differentiable on $X$ with Lipschitz continuous gradient with Lipschitz constant $L_g>0$. Then $\partial_{lsc}^\R  g(x) \neq\emptyset$ for all $x\in X$. Moreover, for $x\in X$, any $(a,u)\in\partial_{lsc}^\R  g(x)$ satisfies $a+\frac{L_g}{2} \geq 0$ and $u-2ax = \nabla g(x)$.
\end{lemma}
\begin{proof}
By \cite[Proposition 3.6]{syga2019global} and \cite[Proposition 2]{bednarczuk2023convergence} which infers $\partial_{lsc}^{\geq} g (x) \neq\emptyset$ for all $x\in X$. By Proposition \ref{proP:121}, $\lsc^{\geq} \subset \lsc^{\R}$, so $\partial_{lsc}^\R g (x)\neq \emptyset$ for all $x\in X$. For the second assertion,
let $x\in X$. Since $g$ is Fr{\' e}chet-differentiable at $x\in X$ with Lipschitz continuous gradient, we have, for all $y\in X$ 
\begin{equation}
\label{eq: grad and subgrad 2}
    \frac{L_g}{2}\Vert y-x \Vert^2 +\langle\nabla g(x),y-x\rangle \geq g(y) -g(x),
\end{equation}
which is the well-known descent Lemma \cite[Lemma 2.64]{Bau2011}.

On the other hand, by taking $(a,u)\in\partial_{lsc}^\R  g(x)$, we have
\begin{equation}
    \label{eq: grad and subgrad 1}
    (\forall y\in X) \qquad g(y) -g(x) \geq -a \left( \Vert y\Vert^2 - \Vert x\Vert^2\right) +\langle u,y-x\rangle.
\end{equation}
Combining \eqref{eq: grad and subgrad 1}, \eqref{eq: grad and subgrad 2} and separating $x$ and $y$,
\begin{equation*}
    \left(\frac{L_g}{2} +a\right)\Vert y \Vert^2 +\langle \nabla g(x) - L_g x-u,y\rangle \geq \left(\frac{L_g}{2} +a\right)\Vert x \Vert^2 +\langle \nabla g(x) - L_g x-u, x\rangle.
\end{equation*}
As this holds for all $y\in X$, $x$ is the global minimizer of the function 
\begin{equation*}
h(z) = \left(\frac{L_g}{2} +a\right)\Vert z \Vert^2 +\langle \nabla g(x) - L_g x-u,z\rangle, 
\end{equation*}
 which is quadratic. This implies that $\frac{L_g}{2}+a \geq 0$ and
\begin{equation*}
    \nabla h (x) = 0 \Leftrightarrow u-2a x = \nabla g(x).
\end{equation*}
Since $(a,u)\in\partial_{lsc}^\R  g(x)$ is taken arbitrarily, the above inequality holds for all the elements in $\partial_{lsc}^\R  g(x)$.
\end{proof}

With the result obtained in Lemma \ref{lem: lsc subdiff and grad}, we can improve algorithm \ref{alg:FB-example} with the following algorithm \ref{alg:FB-smooth g-example}.

\begin{algorithm}[H]\caption{$\lsc^\R$-Forward-Backward Algorithm with $g$ as in Lemma \ref{lem: lsc subdiff and grad}}\label{alg:FB-smooth g-example}
\begin{enumerate}
    \item \textbf{Initialize:} $\gamma >0$, $x_0 \in \dom f$ and $\left( a_0, (\frac{1}{\gamma}+2a_0)x_0\right) \in J_\gamma (x_0), (a_0^g,u_0^g) \in \partial_{lsc}^\R g(x_0)$
    \item \textbf{For $n>0$, update}
    \begin{itemize}
        \item Pick $\left( a_n, (\frac{1}{\gamma}+2a_n)x_n\right) \in J_\gamma (x_n), (a_n^g,u_n^g) \in \partial_{lsc}^\R g(x_n)$
        \item $x_{n+1} \in \argmind{z\in X}{ f(z) + \langle \nabla g(x_n),z\rangle +\left( \frac{1}{2\gamma} +a_n-a_n^g\right) \Vert z-x_n\Vert^2}$
    \end{itemize}
\end{enumerate}
\end{algorithm}
\vspace{0.5cm}

Below, we give the estimation for the behavior of the objective functions for \eqref{eq:forward-backward-operator 1}.
\begin{proposition}
\label{prop:FB f+g norm}
Let $f:X\to (-\infty,+\infty]$ be a proper $\lsc^\R$-convex function and $g:X\to (-\infty,+\infty]$ be a proper Fr{\'e}chet differentiable on $X$ with Lipschitz continuous gradient $L_g>0$. Let \eqref{eq:well-defined FB} hold and $\left( x_n\right)_{n\in\N}$ be a sequence generated by \eqref{eq:forward-backward-operator 1} with stepsize $\gamma>0$. 
For all $y\in X$ and $n\in\N$,
we have
\begin{align}
    (f+g)(y) -(f+g)(x_{n+1}) &\geq \left(\frac{1}{2\gamma}+a_{n+1}\right)\left\Vert y-x_{n+1}\right\Vert ^{2}-\left(\frac{1}{2\gamma}+a_{n}\right) \left\Vert y-x_{n}\right\Vert ^{2} \nonumber\\
    & +\left(\frac{1}{2\gamma}+a_{n}-a_n^g -\frac{L_g}{2} \right)  \left\Vert x_{n+1}-x_{n}\right\Vert ^{2}, \label{eq:prox alg f+g norm}
\end{align}
where $\left(a_{n},\left(\frac{1}{\gamma}+2a_{n}\right)x_{n}\right)\in J_\gamma \left(x_{n}\right)$,
$\left(a_{n+1},\left(\frac{1}{\gamma}+2a_{n+1}\right)x_{n+1}\right)\in J_\gamma \left(x_{n+1}\right)$, and  ${(a_n^g,u_n^g )\in\partial_{lsc}^\R g(x_n)}$.
Moreover, if $\frac{1}{\gamma}+a_n+a_{n+1} \geq a_n^g +\frac{L_g}{2}$, then $(f+g)(x_n)\geq (f+g)(x_{n+1})$.
\end{proposition}

\begin{proof}
By the definition of the updated in \eqref{eq:forward-backward-operator 1}, there exist $\left(a_{n},\left(\frac{1}{\gamma}+2a_{n}\right)x_{n}\right)\in J_\gamma \left(x_{n}\right)$,
$\left(a_{n+1},\left(\frac{1}{\gamma}+2a_{n+1}\right)x_{n+1}\right)\in J_\gamma \left(x_{n+1}\right)$ and $(a_n^g,u_n^g)\in\partial_{lsc}^\R  g(x_n)$ with $a_n^g+\frac{L_g}{2} \geq 0$ such that
\begin{equation}
\label{eq: FB subdifferentials 1}
    \left(a_{n},\left(\frac{1}{\gamma}+2a_{n}\right)x_{n}\right)- \left(a_{n+1},\left(\frac{1}{\gamma}+2a_{n+1}\right)x_{n+1}\right) - (a_n^g,u_n^g) \in \partial_{lsc}^\R  f(x_{n+1}).
\end{equation}
Using the definition of $\lsc^\R$-subdifferentials of $f$, we have, for all $y\in X$
\begin{align*}
f\left(y\right)-f\left(x_{n+1}\right) & \geq-\left(a_{n}-a_{n}^{g}-a_{n+1}\right)\left(\left\Vert y\right\Vert ^{2}-\left\Vert x_{n+1}\right\Vert ^{2}\right)\\
 & +\left\langle \left(\frac{1}{\gamma}+2a_{n}\right)x_{n}-u_{n}^{g}-\left(\frac{1}{\gamma}+2a_{n+1}\right)x_{n+1},y-x_{n+1}\right\rangle .
\end{align*}
We proceed similarly as in Proposition \ref{prop:prox f norm} to obtain 
\begin{align}
f\left(y\right)-f\left(x_{n+1}\right) & \geq\left(\frac{1}{2\gamma}+a_{n+1}\right)\left\Vert y-x_{n+1}\right\Vert ^{2}-\left(\frac{1}{2\gamma}+a_{n}\right)\left\Vert y-x_{n}\right\Vert ^{2} \nonumber\\
 & +\left(\frac{1}{2\gamma}+a_{n}\right)\left\Vert x_{n+1}-x_{n}\right\Vert ^{2} \nonumber\\
 \label{eq:FB f+g from prox}
 & +a_{n}^{g}\left(\left\Vert y\right\Vert ^{2}-\left\Vert x_{n+1}\right\Vert ^{2}\right)-\left\langle u_{n}^{g},y-x_{n+1}\right\rangle .
\end{align}
Let us focus on the last terms on the right hand side of \eqref{eq:FB f+g from prox}
\begin{align}
-a_{n}^{g}\left(\left\Vert y\right\Vert ^{2}-\left\Vert x_{n+1}\right\Vert ^{2}\right)+\left\langle u_{n}^{g},y-x_{n+1}\right\rangle  & =-a_{n}^{g}\left(\left\Vert y\right\Vert ^{2}-\left\Vert x_{n}\right\Vert ^{2}\right)+\left\langle u_{n}^{g},y-x_{n}\right\rangle \nonumber\\
 & +a_{n}^{g}\left(\left\Vert x_{n+1}\right\Vert ^{2}-\left\Vert x_{n}\right\Vert ^{2}\right)+\left\langle u_{n}^{g},x_{n}-x_{n+1}\right\rangle \nonumber\\
 & \leq g\left(y\right)-g\left(x_{n}\right)+a_{n}^{g}\left(\left\Vert x_{n+1}\right\Vert ^{2}-\left\Vert x_{n}\right\Vert ^{2}\right)+\left\langle u_{n}^{g},x_{n}-x_{n+1}\right\rangle \nonumber\\
 & =g\left(y\right)-g\left(x_{n}\right)+a_{n}^{g}\left\Vert x_{n+1}-x_{n}\right\Vert ^{2} \nonumber\\
 & +\left\langle u_{n}^{g}-2a_{n}^{g}x_{n},x_{n}-x_{n+1}\right\rangle .\label{eq:FB subgrad of g}
\end{align}
Plugging \eqref{eq:FB subgrad of g} back into \eqref{eq:FB f+g from prox} we get
\begin{align}
\left(f+g\right) \left(y\right)-f\left(x_{n+1}\right)- g(x_n) & \geq\left(\frac{1}{2\gamma}+a_{n+1}\right)\left\Vert y-x_{n+1}\right\Vert ^{2}-\left(\frac{1}{2\gamma}+a_{n}\right)\left\Vert y-x_{n}\right\Vert ^{2} \nonumber\\
 & +\left(\frac{1}{2\gamma}+a_{n} - a_n^g\right)\left\Vert x_{n+1}-x_{n}\right\Vert ^{2} \nonumber\\
 & -\left\langle u_{n}^{g}-2a_{n}^{g}x_{n},x_{n}-x_{n+1}\right\rangle . \label{eq:FB f+g norm prox 1}
\end{align}

On the other hand, since $g$ is Fr{\'e}chet differentiable with Lipschitz continuous gradient, we apply Lemma \ref{lem: lsc subdiff and grad}
\begin{equation}
\label{eq: descent lemma xn xn+1 FB}
\left\langle u_{n}^{g}-2a_{n}^{g}x_{n},x_{n}-x_{n+1}\right\rangle =\left\langle \nabla g\left(x_{n}\right),x_{n}-x_{n+1}\right\rangle \leq g\left(x_{n}\right)-g\left(x_{n+1}\right)+\frac{L_{g}}{2}\left\Vert x_{n+1}-x_{n}\right\Vert ^{2}.
\end{equation}
Putting \eqref{eq: descent lemma xn xn+1 FB} back into \eqref{eq:FB f+g norm prox 1} to obtain
\begin{align*}
\left(f+g\right)\left(y\right)-\left(f+g\right)\left(x_{n+1}\right) & \geq\left(\frac{1}{2\gamma}+a_{n+1}\right)\left\Vert y-x_{n+1}\right\Vert ^{2}-\left(\frac{1}{2\gamma}+a_{n}\right)\left\Vert y-x_{n}\right\Vert ^{2}\\
 & +\left(\frac{1}{2\gamma}+a_{n}-a_{n}^{g}-\frac{L_{g}}{2}\right)\left\Vert x_{n+1}-x_{n}\right\Vert ^{2},
\end{align*}
which is \eqref{eq:prox alg f+g norm}.
On the other hand, letting $y=x_n$ in \eqref{eq:prox alg f+g norm}, 
\begin{equation}
\left(f+g\right)\left( x_n \right)-\left(f+g\right)\left(x_{n+1}\right)\geq \left(\frac{1}{\gamma}+a_{n+1} +a_n -a_n^g -\frac{L_g}{2}\right)\left\Vert x_n -x_{n+1}\right\Vert ^{2}. \label{eq: proof f+g xn xn+1}
\end{equation}
The RHS of \eqref{eq: proof f+g xn xn+1} is non-negative when $\frac{1}{\gamma}+a_n+a_{n+1} \geq a_n^g+\frac{L_g}{2}$. This conclude the proof. 
\end{proof}

The results obtained in Proposition \ref{prop:FB f+g norm} are similar to the one obtained for the Forward-Backward in \cite[Corollary 3]{bednarczuk2023convergence}.
In analogy to Proposition \ref{prop:prox f norm}, Proposition \ref{prop:FB f+g norm} gives us a crucial estimate of the \eqref{eq:forward-backward-operator 1} algorithm which contributes to the convergence result below.

\begin{theorem}
\label{thm:Fejer FB}
Let $f:X\to (-\infty,+\infty]$ be a proper $\lsc^\R$-convex function and $g:X\to (-\infty,+\infty]$ be a proper Fr\'{e}chet differentiable function on $X$ with Lipschitz continuous gradient with Lipschitz constant $L_g>0$. Let $\left( x_n\right)_{n\in\N}$, $(a_n)_{n\in\N}$, and $(a_n^g)_{n\in\N}$ be sequences generated by \eqref{eq:forward-backward-operator 1} with stepsize $\gamma>0$. Assume that
${\frac{1}{\gamma}+a_n+a_{n+1}\geq a_n^g +\frac{L_g}{2}}$ for all $n\in\N$ where
\begin{equation*}
{\left(a_{n},\left(\frac{1}{\gamma}+2a_{n}\right)x_{n}\right)\in J_\gamma \left(x_{n}\right)},
\quad (a_n^g,u_n^g) \in \partial_{lsc}^\R  g(x_n). 
\end{equation*}
 
We have the following
\begin{enumerate}
    \item If there exists $n_0\in\N$ such that $\frac{1}{2\gamma}+a_{n_0} = a_{n_0}^g +\frac{L_g}{2}=0$ then $x_{n_0+1}$ is the global minimizer.
    \item If $\alpha_n =\frac{1}{2\gamma}+a_n >0$ and $\beta_n = \frac{1}{2\gamma}+a_n -a_n^g -\frac{L_g}{2} > 0$ for all $n\in\N$, Theorem \ref{thm: general convergence results} holds and $\lim_{n\to\infty} (f+g)(x_n) = \inf_{x\in X} (f+g)(x)$.
\end{enumerate}
\end{theorem}
\begin{proof}
The proof follows in the same manner as in Theorem \ref{thm: general convergence results} and Theorem \ref{thm:Fejer prox pt}.
\end{proof}


\section{Projected Subgradient for $\lsc^\R$-convex function}
\label{sec: projected sub alg}
We saw in the previous section thatthe  \eqref{eq:forward-backward-operator 1} algorithm can be written with the $\lsc^\R$-proximal operator. In this section, we investigate the case of the $\lsc^\R$-proximity operator of the indicator function of a closed convex set. This is equivalent to solving the constrainted problem
\begin{equation}
\label{prob: constrainted prob}
    \min_{x\in C} f(x),
\end{equation}
where $C\subset X$ is a closed convex set in Hilbert space and $f:X\to (-\infty,+\infty]$ is proper $\lsc^\R$-convex. We can rewrite this problem in the form
\begin{equation*}
    \min_{x\in X} f(x) +\iota_C (x).
\end{equation*}
To solve \eqref{prob: constrainted prob}, we propose $\lsc^\R$-projected subgradient algorithm which is formally given in Algorithm \ref{alg:PSG-example}. 

\begin{algorithm}[H]\caption{$\lsc^\R$-Projected Subgradient Algorithm}\label{alg:PSG-example}
\textbf{Initialize:} $x_0 \in \dom f$ \\
{\textbf{Set:} $(\gamma_n)_{n\in\N}$ positive\\
\textbf{Compute}: $\left(a_n,\left(\frac{1}{\gamma_n} +2a_n\right) x_n\right)\in J_{\gamma_n} (x_n)$ and $(a_n^f,u_n^f) \in \partial_{lsc}^\R  f(x_n)$ such that
\begin{equation*}
    2\gamma_n(a_n -a_n^f) >-1
\end{equation*}}\\
{\normalfont\textbf{Update: }}{$\displaystyle{x_{n+1} = \Proj_C \left( \frac{(1+2\gamma_n a_n) x_n - \gamma_n u_n^f}{1+ 2\gamma_n (a_n-a_n^f) } \right)}$
}\\
\textbf{Return: $x_{n+1}$}
\end{algorithm}
\vspace{0.5cm}

Further details about the algorithm and the variable stepsize $\gamma_n$ are given in \eqref{eq: projected subgradient explicit form} and in Theorem \ref{thm: PSG general case convergence} below.
In the following proposition, we show that the $\lsc^\R$-proximal operator of indicator function of $C$ coincides with the projection onto $C$ without convexity assumption.

\begin{proposition}
\label{prop: lsc prox projection}
Let $C$ be a closed subset of $X$ with more than one element and let $x\in X, \gamma>0$. If $(-\tfrac1{2\gamma},0)\notin J_\gamma (x)\cap \mathrm{ran } (J_\gamma+\partial_{lsc}^\R  \iota_C)$ then we have 
\[
\left(J_\gamma+\partial_{lsc}^\R \iota_{C}\right)^{-1} J_\gamma\left(x\right) =\Proj_{C}\left(x\right).
\]
\end{proposition}
\begin{proof}
Recall that
\begin{equation}
\label{eq: prox indicator C}
    J_\gamma (x)\cap \mathrm{ran } (J_\gamma+\partial_{lsc}^\R  \iota_C) =\{ \phi\in J_\gamma (x): \exists x^+ \in X \ \text{s.t. } \phi \in J_\gamma (x^+)+\partial_{lsc}^\R  \iota_C (x^+) \}.
\end{equation}
Let $x^{+}\in\left(J_\gamma+\partial_{lsc}^\R \iota_{C}\right)^{-1} J_\gamma\left(x\right)$. By assumption, \eqref{eq: prox indicator C} implies that there exist $(a,(1/\gamma+2a)x) \in J_\gamma(x)$ and
${ (a^+,(1/\gamma+2a^+ )x^+) \in J_\gamma(x^+)}$  such that
\begin{align}
\label{eq: prox indicator explicit form}
\left(a-a^{+},\left(\frac{1}{\gamma}+2a\right)x-\left(\frac{1}{\gamma}+2a^{+}\right)x^{+}\right) & \in\partial_{lsc}^\R \iota_{C}\left(x^{+}\right), \quad 2\gamma a > -1, \ 2\gamma a^{+} \geq -1.
\end{align}
If $x^{+}\notin C$ then $\partial_{lsc}^\R \iota_{C}\left(x^{+}\right) = \emptyset$, which contradicts the above inclusion. Hence, it must be $x^{+}\in C$. By the definition of $\Phi_{lsc}^\R$-subdifferentials, for any $y\in X$, we have 
\begin{equation}
\iota_{C}\left(y\right)-\iota_{C}\left(x^{+}\right)\geq-\left(a-a^{+}\right)\left(\left\Vert y\right\Vert ^{2}-\left\Vert x^{+}\right\Vert ^{2}\right)+\left\langle \left(\frac{1}{\gamma}+2a\right)x-\left(\frac{1}{\gamma}+2a^{+}\right)x^{+},y-x^{+}\right\rangle .
\label{eq: subgrad iota C}
\end{equation}

By taking $y\in C$, \eqref{eq: subgrad iota C} reads 
\begin{align}
\left(a-a^{+}\right)\left(\left\Vert y\right\Vert ^{2}-\left\Vert x^{+}\right\Vert ^{2}\right) & \geq\left\langle \left(\frac{1}{\gamma}+2a\right)x-\left(\frac{1}{\gamma}+2a^{+}\right)x^{+},y-x^{+}\right\rangle .\label{eq: subgrad iota C1}
\end{align}
Further simplifying \eqref{eq: subgrad iota C1}
\begin{align}
\left(a-a^{+}\right)\left\Vert y-x^+\right\Vert ^{2} & \geq \left(\frac{1}{\gamma}+2a\right) \left\langle x-x^+ ,y-x^{+}\right\rangle .\label{eq: subgrad iota C2}
\end{align} 
From \eqref{eq: subgrad iota C2}, we have 
\begin{equation*}
    \left(a-a^{+}\right)\left\Vert y-x^+\right\Vert ^{2}  \geq \left(\frac{1}{2\gamma} +a\right)\left\Vert y-x^+\right\Vert ^{2} + \left(\frac{1}{2\gamma} + a\right)\left\Vert x-x^+\right\Vert ^{2}-\left( \frac{1}{2\gamma}+a\right)\left\Vert y-x\right\Vert ^{2},
\end{equation*}
so that 
\begin{equation}
    \left( \frac{1}{2\gamma}+a\right)\left\Vert y-x\right\Vert ^{2} \geq \left(\frac{1}{2\gamma} +a^+\right)\left\Vert y-x^+\right\Vert ^{2} + \left(\frac{1}{2\gamma} + a\right)\left\Vert x-x^+\right\Vert ^{2}.
\end{equation}
As $2\gamma a > -1, 2\gamma a^+ \geq -1$, all the coefficients are nonnegative, this infers
\begin{equation}
    \label{eq: norm square indicator}
    \left( \frac{1}{2\gamma}+a\right)\left\Vert y-x\right\Vert ^{2} \geq  \left(\frac{1}{2\gamma} + a\right)\left\Vert x-x^+\right\Vert ^{2},
\end{equation}
for all $y\in C$. By \eqref{eq: prox indicator explicit form}, $x^+ \in \Proj_C (x)$.

On the other hand, let us assume that $x^{+} \in \Proj_{C}\left(x\right)$, we have 
\begin{equation}
    \label{eq: projection property}
    (\forall y\in C) \qquad \Vert y-x\Vert \geq \Vert x^+ - x\Vert.
\end{equation}
Taking square both sides of \eqref{eq: projection property} and multiply with $\frac{1}{2\gamma} +a >0$ for some $a \in\R$, we obtain
\begin{equation*}
    \left(\frac{1}{2\gamma} +a \right) \Vert y-x\Vert^2 \geq  \left(\frac{1}{2\gamma} +a \right) \Vert x^+-x\Vert^2.
\end{equation*}
This is equivalent to 
\begin{equation}
    \label{eq:projection property 2}
    0 \geq - \left(\frac{1}{2\gamma} +a \right) (\Vert y\Vert^2 -\Vert x^+\Vert^2) +\langle \left(\frac{1}{\gamma} +2a \right)x, y-x^+\rangle. 
\end{equation}
Because $x^+\in C$, \eqref{eq:projection property 2} implies that $\left(\frac{1}{2\gamma} +a, \left(\frac{1}{\gamma} +2a \right)x \right)\in\partial_{lsc}^\R  \iota_C (x^+)$.
We can write 
\begin{equation}
    \left(\frac{1}{2\gamma} +a, \left(\frac{1}{\gamma} +2a \right)x \right) = \left( a- \left(-\frac{1}{2\gamma}\right), \left(\frac{1}{\gamma} +2a \right)x - \left(\frac{1}{\gamma} -\frac{1}{\gamma} \right)x^+\right)\in J_\gamma (x) -J_\gamma (x^+).
\end{equation}
This infers $\left(a, \left(\frac{1}{\gamma} +2a \right)x \right) \in (\partial_{lsc}^\R  \iota_C +J_\gamma)(x^+)$, so that $x^+ \in \proxlsc_{\gamma \iota_C}(x)$.
\end{proof}
\begin{remark}
    The assumption $(-\tfrac1{2\gamma},0)\notin J_\gamma (x)\cap \mathrm{ran } (J_\gamma+\partial_{lsc}^\R  \iota_C)$ also implies the uniqueness of the $\lsc^\R$-proximal operator of an indicator function. Assume that there are $x_1,x_2 \in \left(J_\gamma+\partial_{lsc}^\R \iota_{C}\right)^{-1} J_\gamma\left(x\right)$. Following \eqref{eq: norm square indicator} in the proof of Proposition \ref{prop: lsc prox projection}, we have
    \begin{equation*}
        \Vert x-x_2\Vert^2 \geq \Vert x-x_1\Vert^2 \geq \Vert x-x_2\Vert^2,
    \end{equation*}
so $x_1=x_2$.
\end{remark}
Let us go back to problem \eqref{prob: constrainted prob}. Since $C$ is closed, the function $\iota_C$ is lsc and so it is $\lsc^\R$-convex (by virtue of \cite[Proposition 6.3]{Rub2013}).
Motivated by \eqref{eq:forward-backward-operator 1} and Proposition \ref{prop: lsc prox projection}, we propose $\lsc^\R$-Projected Subgradient Algorithm,
\begin{align}
    x_{n+1} &\in \Proj_C \left(J_\gamma^{-1} \left(J_\gamma-\partial_{lsc}^\R  f\right)\left(x_{n}\right)\right). \label{eq: projected subgradient}\tag{$\lsc^\R$-PSG}
\end{align}
Observe that if ${\mathrm{ran } J_\gamma \cap \mathrm{ran } (J_\gamma-\partial_{lsc}^\R  f) \neq \emptyset}$, algorithm \eqref{eq: projected subgradient} is well-defined i.e. for every $n\in\N$, there exists $x_{n+1}$ such that \eqref{eq: projected subgradient} holds.

In fact, ${\mathrm{ran } J_\gamma \cap \mathrm{ran } (J_\gamma-\partial_{lsc}^\R  f)}$ is always nonempty. Let us prove this by taking $x\in \dom \partial_{lsc}^\R  f$, $(a_f,u_f) \in \partial_{lsc}^\R  f(x)$ and $(a,(\frac{1}{\gamma}+2a)x) \in J_\gamma(x)$ with $\gamma>0$ and $2a\gamma\geq -1$. Then 
\begin{align}
    J_\gamma^{-1} \left(a-a_f, \left(\frac{1}{\gamma}+2a \right)x - u_f \right) & \subseteq J_\gamma^{-1} \left( J_\gamma-\partial_{lsc}^\R  f \right)(x) . \label{eq: J-1 J-f}
\end{align}
By formula \eqref{eq: subgrad of J_gamma} in Example \ref{ex:1 J function}, the domain of $J_\gamma^{-1}$ is nonempty when $a-a_f \geq -1/(2\gamma)$. While $-1/(2\gamma) \leq a$, we can take $a$ large enough so that $a-a_f >-1/(2\gamma)$. Again, by formula \eqref{eq: subgrad of J_gamma}, $J_\gamma^{-1}$ is a single valued map and we obtain
\begin{equation*}
    J_\gamma^{-1} \left( J_\gamma-\partial_{lsc}^\R  f \right)(x) 
    = \frac{\left(\frac{1}{\gamma}+2a \right)x - u^f}{\frac{1}{\gamma}+ 2(a - a^f) } 
    =\frac{(1+2\gamma a) x - \gamma u^f}{1+ 2\gamma (a -a^f) }.
\end{equation*}
On the other hand, when
\begin{equation}
a-a_f = -\frac{1}{2\gamma},\text{ and } \left( \frac{1}{\gamma} +2a\right)x-u_f =0, \label{eq: psg-critical point special case}
\end{equation}
then by \eqref{eq: subgrad of J_gamma}
\[
J_\gamma^{-1} \left(a-a_f, \left(\frac{1}{\gamma}+2a \right)x - u_f \right) = X.
\]
Moreover, \eqref{eq: psg-critical point special case} implies that $x$ is a $a + 1/(2\gamma)$-critical point of $f$ (see Definition \ref{def: lsc stationary}). 

With the well-defined inner operator of \eqref{eq: projected subgradient}, we have
\begin{equation*}
    J_\gamma^{-1} \left( J_\gamma-\partial_{lsc}^\R  f \right)(x_n) 
    = \frac{\left(\frac{1}{\gamma}+2a_n \right)x_n - u_n^f}{\frac{1}{\gamma}+ 2(a_n-a_n^f) } 
    =\frac{(1+2\gamma a_n) x_n - \gamma u_n^f}{1+ 2\gamma (a_n-a_n^f) },
\end{equation*}
with $\gamma>0, 2\gamma(a_n -a_n^f)> -1$.
\eqref{eq: projected subgradient} takes an explicit form
\begin{equation}
    x_{n+1} \in \Proj_C \left(J_\gamma^{-1} \left(J_\gamma-\partial_{lsc}^\R  f\right)\left(x_{n}\right)\right) = \Proj_C \left( \frac{(1+2\gamma a_n) x_n - \gamma u_n^f}{1+ 2\gamma (a_n-a_n^f) } \right). \label{eq: projected subgradient explicit form}
\end{equation}



Let us state an estimate related to the objective function of the $\lsc^\R$-PSG algorithm to problem \eqref{prob: constrainted prob} with the stepsize $\gamma>0$ being kept constant.
\begin{lemma}
\label{lem: PSG estimation-1}
Let $f:X\to (-\infty,+\infty]$ be a proper $\lsc^\R$-convex function and $C$ be a closed convex set. Let $(x_n)_{n\in\N}$ be the sequence generated by \eqref{eq: projected subgradient explicit form}, then the following holds for any $x\in C$
\begin{align}
   (1+2\gamma (a_n-a_n^f)) \Vert x-x_{n+1} \Vert^2 &\leq  \left(1+2\gamma a_n\right) \left\Vert x-x_{n}\right\Vert ^{2}+\frac{\gamma^{2}}{1+2\gamma\left(a_{n}-a_{n}^{f}\right)} \left\Vert 2a_{n}^{f}x_{n}-u_{n}^{f}\right\Vert ^{2} \nonumber\\
     & +2\gamma \left[f\left(x\right)-f\left(x_{n}\right)\right], \label{eq: PSG final norm}
\end{align}
where 
\begin{equation*}
    \left( a_n , \left( \frac{1}{\gamma}+2a_n\right) x_n \right) \in J_\gamma (x_n) ,\quad (a_n^f, u_n^f) \in \partial_{lsc}^\R  f(x_n), \quad 2\gamma (a_n -a_n^f) >-1 \ \forall n\in\N.
\end{equation*}

\end{lemma}
\begin{proof}
Let $(x_n)_{n\in\N}$ be a sequence from \eqref{eq: projected subgradient explicit form} and $x\in C$ so that $x= \Proj_C(x)$. We consider
\begin{equation*}
    \Vert x-x_{n+1} \Vert^2 = \left\Vert \Proj_C (x) - \Proj_C \left( \frac{(1+2\gamma a_n) x_n - \gamma u^f_n}{1+ 2\gamma (a_n-a^f_n)} \right) \right\Vert^2 .
\end{equation*}
Since $C$ is closed and convex, the projection operator $\Proj_C$ is nonexpansive. Together with \eqref{eq: J-1 J-f}, we continue the above equality

\begin{align}
    \Vert x-x_{n+1} \Vert^2 &\leq \left\Vert x - \frac{(1+2\gamma a_n )x_n - \gamma u_n^f}{1+ 2\gamma (a_n -a_n^f) } \right\Vert^2 = \left\Vert x - x_n -\frac{2\gamma a_n^f x_n - \gamma u_n^f}{1+ 2\gamma (a_n -a_n^f) } \right\Vert^2 \nonumber\\
    & =\left\Vert x-x_{n}\right\Vert ^{2}+\frac{\gamma^{2}}{\left(1+2\gamma\left(a_{n}-a_{n}^{f}\right)\right)^{2}}\left\Vert 2a_{n}^{f}x_{n}-u_{n}^{f}\right\Vert ^{2} \nonumber\\
     & +\frac{2\gamma}{1+2\gamma\left(a_{n}-a_{n}^{f}\right)}\left\langle x-x_{n},u_{n}^{f}-2a_{n}^{f}x_{n}\right\rangle. \nonumber
\end{align}
By using the definition of $\lsc^\R$-subdifferentials, we obtain
\begin{align}
     \Vert x-x_{n+1} \Vert^2 & \leq\left\Vert x-x_{n}\right\Vert ^{2}+\frac{\gamma^{2}}{\left(1+2\gamma\left(a_{n}-a_{n}^{f}\right)\right)^{2}}\left\Vert 2a_{n}^{f}x_{n}-u_{n}^{f}\right\Vert ^{2} \nonumber\\
     & +\frac{2\gamma}{1+2\gamma\left(a_{n}-a_{n}^{f}\right)}\left[f\left(x\right)-f\left(x_{n}\right)+a_{n}^f\left(\left\Vert x\right\Vert ^{2}-\left\Vert x_{n}\right\Vert ^{2}\right)-2a_{n}^{f}\left\langle x-x_{n},x_{n}\right\rangle \right] \nonumber\\
     & =\left(1+\frac{2\gamma a_{n}^{f}}{1+2\gamma\left(a_{n}-a_{n}^{f}\right)}\right)\left\Vert x-x_{n}\right\Vert ^{2}+\frac{\gamma^{2}}{\left(1+2\gamma\left(a_{n}-a_{n}^{f}\right)\right)^{2}}\left\Vert 2a_{n}^{f}x_{n}-u_{n}^{f}\right\Vert ^{2} \nonumber\\
     & +\frac{2\gamma}{1+2\gamma\left(a_{n}-a_{n}^{f}\right)}\left[f\left(x\right)-f\left(x_{n}\right)\right]. \label{eq: PSG norm estimation}
\end{align}
From the assumptiom, $1+2\gamma (a_n -a_n^f)>0$, we can further simplify \eqref{eq: PSG norm estimation} to obtain \eqref{eq: PSG final norm}.
\end{proof}
From Lemma \ref{lem: PSG estimation-1}, we notice that $\left(1+\frac{2\gamma a_{n}^{f}}{1+2\gamma\left(a_{n}-a_{n}^{f}\right)}\right) \geq 0$ as $2\gamma (a_n-a_n^f)> -1$ from the assumption. 
When $a_n^f \leq 0$, then $\frac{1+2\gamma a_n}{1+ 2\gamma (a_n -a_n^f)} \leq 1$, from \eqref{eq: PSG final norm}, we obtain
\begin{align}
    \Vert x-x_{n+1} \Vert^2 & \leq  \frac{1+2\gamma a_n}{1+2\gamma (a_n-a_n^f)} \left\Vert x-x_{n}\right\Vert ^{2} +\left(\frac{\gamma}{1+2\gamma\left(a_{n}-a_{n}^{f}\right)} \right)^2 \left\Vert 2a_{n}^{f}x_{n}-u_{n}^{f}\right\Vert ^{2} \nonumber\\
    & +\frac{2\gamma}{1+2\gamma\left(a_{n}-a_{n}^{f}\right)} \left[ f(x) -f(x_n)\right] \nonumber\\
    &\leq \left\Vert x-x_{n}\right\Vert ^{2} +\left(\frac{\gamma}{1+2\gamma\left(a_{n}-a_{n}^{f}\right)} \right)^2 \left\Vert 2a_{n}^{f}x_{n}-u_{n}^{f}\right\Vert ^{2}  +\frac{2\gamma  \left[ f(x) -f(x_n)\right]}{1+2\gamma\left(a_{n}-a_{n}^{f}\right)}.
    \label{eq: PSG norm estimation with f}
\end{align}
The above expression is similar to the ones obtained for the class of subgradient descent algorithms in the convex case \cite{boyd2003subgradient}. Let us denote $S$ is the set of minimiser of problem \eqref{prob: constrainted prob}. To proceed with the convergence analysis, we assume the followings.

\begin{assumption}
\label{ass: PSG}
    For $n\in\N,(a_n^f,u_n^f) \in \partial_{lsc}^\R  f(x_n)$. Let us consider the following:
\begin{enumerate}[label=(\roman*)]
    \item There exists a constant $U>0$ such that $\Vert 2a_n^f x_n -u_n^f\Vert \leq U$ for all $n\in\N$.
    \item For $\gamma>0$,  $\sum_{n\in\N} \frac{1}{\left[ 1+2\gamma\left(a_{n}-a_{n}^{f}\right) \right]^2}  <+\infty$ and $\sum_{n\in\N} \frac{1}{ 1+2\gamma\left(a_{n}-a_{n}^{f}\right)}  =+\infty$ where\\
    ${(a_n, (\frac{1}{\gamma}+2a_n) x_n) \in J_\gamma (x_n)}$. 
    \item For $(\gamma_n)_{n\in\N}$ positive, $\sum_{n\in\N} \frac{\gamma_n^2}{ 1+2\gamma_n (a_n-a_n^f) }  <+\infty$ and $\sum_{n\in\N} \gamma_n =+\infty$ where\\
    ${(a_n, (\frac{1}{\gamma_n}+2a_n) x_n) \in J_{\gamma_n} (x_n)}$.
\end{enumerate} 
\end{assumption}

Assumption \ref{ass: PSG}-(i) ensures that the quantity $\Vert 2a_n^f x_n -u_n^f\Vert$ is uniformly bounded. 
Assumption \ref{ass: PSG}-(ii) with fixed stepsize is related to the convex subgradient method with square summable but not summable stepsize \cite{alber1998projected}. This ensures the weak convergence of the sequence and $\lim_{n\to\infty} f(x_n) = f(x^*)$. However, in our general setting of $\lsc^\R$-subdiferentials, we obtain convergence results with Assumption \ref{ass: PSG}-(ii) only when $a_n^f \leq 0$ for all $n\in\N, (a_n^f,u_n^f) \in \partial_{lsc}^\R  f(x_n)$. 
\begin{proposition}
\label{prop: PSG convex}
Let $f:X\to (-\infty,+\infty]$ be a proper $\lsc^\R$-convex function and $C$ be a closed convex subset of $X$. Let $\left( x_n\right)_{n\in\N}$ be a sequence generated by \eqref{eq: projected subgradient explicit form} with stepsize $\gamma>0$. Let Assumption \ref{ass: PSG}-(i,ii) hold and $a_n^f \leq 0$ for all $n\in\N, (a_n^f,u_n^f) \in \partial_{lsc}^\R  f(x_n)$. Consider a point $x^*\in S$, then the followings hold
\begin{enumerate}[label=(\roman*)]
    \item $\left\Vert x^{*}-x_{n}\right\Vert ^{2}$
    converges and is bounded.
    \item $ \mathrm{dist}^{2}\left(x_{n},S\right)$ is decreasing and converges.
    \item $\lim_{n\to\infty} f(x_n) = f(x^*)$
\end{enumerate}
\end{proposition}
\begin{proof}
Since $a_n^f \leq 0$ for all $n\in\N$, estimate \eqref{eq: PSG norm estimation with f} holds. Taking $x^*\in S$, we have $f(x^*) \leq f(x_n)$ for all $n\in\N$, and from \eqref{eq: PSG norm estimation with f}, we obtain 
\begin{align}
    \Vert x^*-x_{n+1} \Vert^2 &\leq \left\Vert x^*-x_{n}\right\Vert ^{2} +\left(\frac{\gamma}{1+2\gamma\left(a_{n}-a_{n}^{f}\right)} \right)^2 \left\Vert 2a_{n}^{f}x_{n}-u_{n}^{f}\right\Vert ^{2} \nonumber\\
    & +\frac{2\gamma}{1+2\gamma\left(a_{n}-a_{n}^{f}\right)} \left[ f(x^*) -f(x_n)\right] \nonumber\\
    &\leq \left\Vert x^*-x_{n}\right\Vert ^{2} +\left(\frac{\gamma}{1+2\gamma\left(a_{n}-a_{n}^{f}\right)} \right)^2 U^2
    \label{eq: PSG norm estimation with f af<0}
\end{align}
By Assumption \ref{ass: PSG}-(ii), the last term on RHS of \eqref{eq: PSG norm estimation with f af<0} is summable. By Lemma \ref{lem: monotone sequence notnorm}, the assertions \textit{(i)} and \textit{(ii)} are proved.

For \textit{(iii)}, we consider \eqref{eq: PSG norm estimation with f} again and by taking the finite sum till $N\in\N$, we obtain
\begin{align}
    &\sum_{n=0}^N \frac{2\gamma}{1+2\gamma\left(a_{n}-a_{n}^{f}\right)} \left[ f(x_n)- f(x^*) \right] \nonumber\\
    &\leq \sum_{n=0}^N \left[\left\Vert x^*-x_{n}\right\Vert ^{2} -\Vert x^*-x_{n+1} \Vert^2\right] +\sum_{n=0}^N \left(\frac{\gamma}{1+2\gamma\left(a_{n}-a_{n}^{f}\right)} \right)^2 U^2 \nonumber\\
    & =\left[\left\Vert x^*-x_{0}\right\Vert ^{2} -\Vert x^*-x_{N+1} \Vert^2\right] +\sum_{n=0}^N \left(\frac{\gamma}{1+2\gamma\left(a_{n}-a_{n}^{f}\right)} \right)^2 U^2 \nonumber\\
    & \leq \left\Vert x^*-x_{0}\right\Vert ^{2} +\sum_{n=0}^N \left(\frac{\gamma}{1+2\gamma\left(a_{n}-a_{n}^{f}\right)} \right)^2 U^2,
\end{align}
Letting $N\to\infty$ and using Assumption \ref{ass: PSG}-(ii), we obtain
\begin{equation*}
    \sum_{n=0}^\infty \frac{2\gamma}{1+2\gamma\left(a_{n}-a_{n}^{f}\right)} \left[ f(x_n)- f(x^*) \right] <+\infty.
\end{equation*}
Since $\sum_{n=0}^\infty \frac{2\gamma}{1+2\gamma\left(a_{n}-a_{n}^{f}\right)} = +\infty$, and $f(x_n)\geq f(x^*)$ for all $n\in\N$, we infer
\begin{equation}
    \lim_{n\to\infty} f(x_n) -f(x^*) =0.
\end{equation}
\end{proof}
Observe that we require $a_n^f \leq 0$ in order to obtain the convergence results. This is different to the previous Sections, where a part of the information of the next iterate is known i.e. $a_{n+1}$ from $J_\gamma (x_{n+1})$.
The results obtained in Proposition \ref{prop: PSG convex} are also similar to \cite[Proposition 1 and Lemma 1]{alber1998projected}. In Proposition \ref{prop: PSG convex}, we keep the stepsize $\gamma$ constant so the adaptivity is transfered from the stepsize to $a_n$ in $J_\gamma (x_n)$. 

In general, we do not know the sign of $a_n^f$ for  $n\in\N$. Hence, to obtain convergence results for \eqref{eq: projected subgradient}, we need to restrict the next element $a_{n+1}$ in $J_\gamma (x_{n+1})$. The condition on $a_{n+1}$ can be relaxed by changing the stepsize $\gamma_n$ instead of keeping it constant.

\begin{theorem}
\label{thm: PSG general case convergence}
Let $f:X\to (-\infty,+\infty]$ be proper $\lsc^\R$-convex and $C$ be a nonempty closed convex set. Let $\left( x_n\right)_{n\in\N}$ be a sequence generated by \eqref{eq: projected subgradient explicit form} with positive stepsize $(\gamma_n)_{n\in\N}$. Assume that Assumption \ref{ass: PSG}-(i) and (iii) hold with
\begin{equation}
\label{ass: an - anf an+1}
(\forall n\in\N) \qquad    
2\gamma_n (a_n -a_n^f) \geq 2\gamma_{n+1} a_{n+1},
\end{equation}
where $\left(a_{n},(1/\gamma_n+2a_n)x_n \right)\in J_{\gamma_n} \left(x_{n}\right),(a_n^f,u_n^f) \in \partial_{lsc}^\R  f(x_n)$ and $2\gamma_n (a_n-a_n^f)>-1$. Consider a point $x^*\in S$. Then 
\begin{itemize}
    \item Theorem \ref{thm: general convergence results}-(i,ii) hold for $\alpha_n =1+2\gamma_n a_n, \beta_n = 0$ and $\displaystyle{\liminf_{n\to\infty} f(x_n)-f(x^*) =0}$.
    \item If $\alpha_n$ is non-decreasing then we have Theorem \ref{thm: general convergence results}-(v,vi). 
\end{itemize}
\end{theorem}

\begin{proof}
Let us take estimate \eqref{eq: PSG final norm} in Lemma \ref{lem: PSG estimation-1} with $x=x^*\in S$ with assumption \eqref{ass: an - anf an+1} to obtain 
\begin{align}
    (1+2\gamma_{n+1} a_{n+1}) \Vert x^*-x_{n+1} \Vert^2 &\leq  \left(1+2\gamma_n a_n\right) \left\Vert x^*-x_{n}\right\Vert ^{2}+\frac{\gamma_n^{2}}{1+2\gamma_n\left(a_{n}-a_{n}^{f}\right)} \left\Vert 2a_{n}^{f}x_{n}-u_{n}^{f}\right\Vert ^{2} \nonumber\\
     & +2\gamma_n \left[f\left(x^*\right)-f\left(x_{n}\right)\right]\nonumber\\
     & \leq \left(1+2\gamma_n a_n\right) \left\Vert x^*-x_{n}\right\Vert ^{2}+ \frac{\gamma_n^{2} U^2}{1+2\gamma_n\left(a_{n}-a_{n}^{f}\right)}
     \label{eq: PSG covergence theorem proof-1}
\end{align}
Since the last term on the right hand side is summable, Lemma \ref{lem: monotone sequence notnorm} gives us statement (i)-(ii). For (iii), the proof follows along the same steps as in Proposition \ref{prop: PSG convex} combined with Assumption 1-(iii).
Additionally, when $\alpha_n$ is non-decreasing, we divide \eqref{eq: PSG covergence theorem proof-1} by $1+2\gamma_{n+1} a_{n+1}$ and obtain
\begin{align*}
    \Vert x^*-x_{n+1} \Vert^2 &\leq \frac{1+2\gamma_n a_n}{1+2\gamma_{n+1} a_{n+1}} \left\Vert x^*-x_{n}\right\Vert ^{2}+\frac{\gamma_n^{2} U^2}{(1+2\gamma_{n+1} a_{n+1})(1+2\gamma_n (a_n-a_n^f))}\\ 
    &\leq \left\Vert x^*-x_{n}\right\Vert ^{2}+\frac{\gamma_n^{2}}{1+2\gamma_n (a_n-a_n^f)} \frac{U^2}{1+2\gamma_{0} a_{0}} .
\end{align*}
Combining the above inequality, which is analogous to inequality \eqref{eq: PSG norm estimation with f af<0} of Proposition \ref{prop: PSG convex}, with Assumption \ref{ass: PSG}-(iii), we infer (iv) and (v).
Lastly, we have that 
\[ \liminf_{n\to\infty} f(x_n)-f(x^*) =0,\]
by following the same argument as in \cite[Proposition 2-(i)]{alber1998projected}. 
\end{proof}
\begin{remark}
\label{rmk: const gamma}
When $\gamma$ is kept constant, the restriction \eqref{ass: an - anf an+1} becomes $a_n-a_n^f \geq a_{n+1}$. Combining this with $2\gamma (a_n-a_n^f) >-1$ can cause early stopping when $a_n^f \geq -1/2\gamma$. 
Because
\[
-\frac{1}{2\gamma} \leq a_{n+1}\leq a_n -a_n^f \leq a_0 -\sum_{i=0}^n a_i^f,
\]
so we need $a_0$ big enough to ensure the algorithm converges.
\end{remark}
\begin{remark}
    Instead of taking $x^*\in S$ as the global minimzer, we can choose a point in the set
    \begin{equation}
        \overline{S} :=\left\{ x\in C: f(x) \leq f(x_n), \forall n\in\N \right\}.
    \end{equation}
    Then Proposition \ref{prop: lsc prox projection} and Theorem \ref{thm: PSG general case convergence} still hold for $x^*\in \overline{S}$.
\end{remark}
\section{Numerical Examples}
\label{sec: numerical examples}
In this section, we give some numerical examples for the different algorithm that are analyzed in this paper.

\paragraph{$\lsc^\R$-Proximal Point Algorithm:} We continue Example~\ref{ex:strongly convex function} to illustrate the performance of the $\lsc^\R$-proximal point Algorithm works. We use the starting point $x_0 = -10$ $a_0 =1$ and fix the number of iteration to $N=101$. Thanks to the closed form of $\lsc^\R$-subgradient and $\lsc^\R$-proximal operator of $f$, we can set the update rule for $a_n$ to
\begin{equation*}
    a_{n+1} = a_n +0.9, 
\end{equation*}
which satisfies the condition $a_{n}-a_{n+1} \geq -1$ of $\partial_{lsc}^\R  f(x)$.
The function $f$ has a minimizer at $x=0$ with $f(0)=0$.
We test for multiple values of $\gamma=0.01, 0.1, 1, 10$ and plot the absolute value of the function value, the distance between two iterates and the distance of the iterate to the minimizer in Figure \ref{fig:plots-PPA}).  

\begin{figure}
    \centering
    \begin{tabular}{ccc}
        \makecell{\includegraphics[width=0.33\textwidth, trim ={15  0 0 0}, clip]{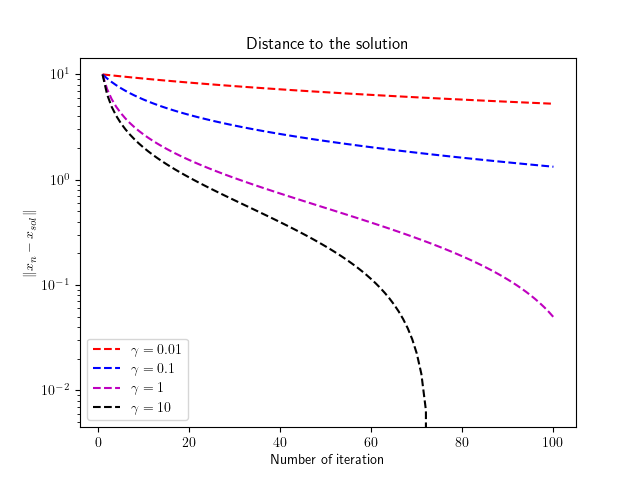}} & 
        \makecell{\includegraphics[width=0.33\textwidth, trim ={15 0 0 0}, clip]{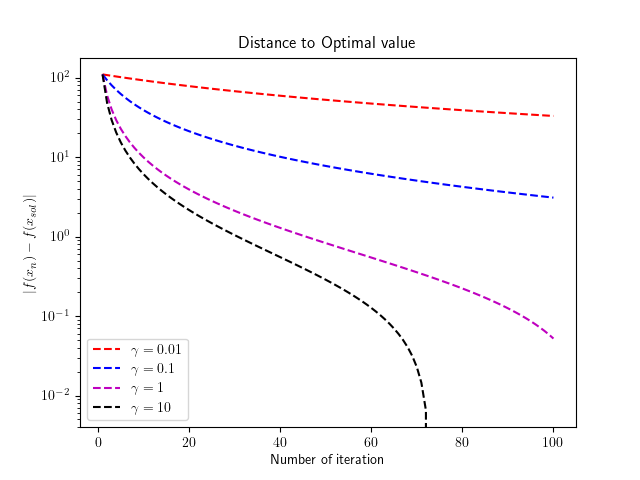}} & \makecell{\includegraphics[width=0.33\textwidth, trim ={15  0 0 0}, clip]{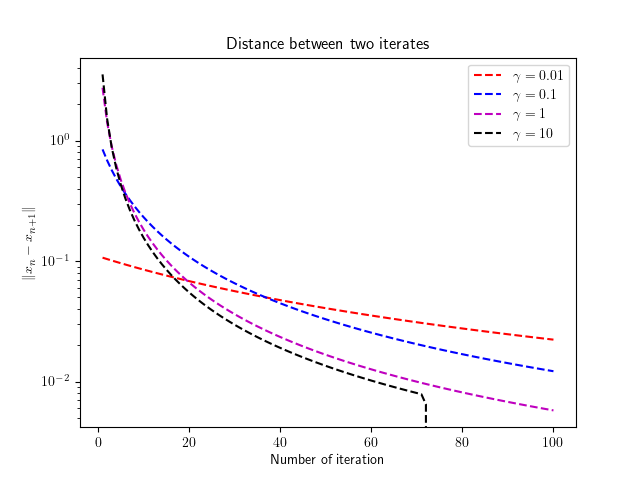}}\\
       \multicolumn{1}{c}{(a)}&  \multicolumn{1}{c}{(b)} & \multicolumn{1}{c}{(c)} 
    \end{tabular}
    \caption{From left to right: distance between the current iteration and the minimizer; distance to the optimal value; distance between two successive iterates.}
    \label{fig:plots-PPA}
\end{figure}

\paragraph{$\lsc^\R$-Projected Subgradient Algorithm:} We illustrate  Algorithm \ref{alg:PSG-example} by solving the following problem
\begin{equation}
\label{eq: quadratic problem}
\min_{x\in C} f(x) = \min_{x\in C} \left\langle x,Q x\right\rangle ,
\end{equation}
where $C = B(0,1)\subset\mathbb{R}^{n}$ is the unit ball and $Q\in\mathbb{R}^{n\times n}$
is a full rank symmetric matrix. We can the problem as
\[
\min_{x\in\mathbb{R}^{n}}\left\langle x,Qx\right\rangle +\iota_{C}\left(x\right),
\]
In Examples \ref{ex:1 J function} and \ref{ex: xQx}, we derived expressions for $J_\gamma,J_\gamma^{-1}$ and the $\lsc^\R$-subgradient of $f$.
We use~\eqref{eq: projected subgradient explicit form}  to state Algorithm \eqref{eq: projected subgradient} in this case as
\begin{align*}
x_{n+1} & =\Proj_{C}\left(\frac{\left[\left(\frac{1}{\gamma}+2a_{n}-2a_{n}^{f}\right)Id-2Q\right]x_{n}}{\frac{1}{\gamma}+2a_{n}-2a_{n}^{f}}\right) 
=\Proj_{C}\left(x_{n}-2\gamma\frac{Qx_{n}}{1+2\gamma a_{n}-2\gamma a_{n}^{f}}\right).
\end{align*}

\begin{example}
\label{ex: numerical ball}
In the case where $C=B(0,1)$, the problem \eqref{eq: quadratic problem} is equivalent to
\begin{equation}
\label{eq: quadratic diagonal problem}
\min_{y\in B(0,1)} \langle y, D y\rangle,
\end{equation}
where $y = Px$ and $D$ is a diagonal matrix of eigenvalues of $Q$ by using matrix decomposition $Q = P D P^{-1}$.
Thanks to Example~\ref{ex: xQx}, we have 
\[ \Vert y\Vert = \sqrt{\Vert y\Vert^2} = \sqrt{\langle Px, Px\rangle} = \Vert x\Vert \leq 1.\]
As $D$ is a diagonal matrix, problem \eqref{eq: quadratic diagonal problem} is
\begin{equation*}
    \min_{y\in B(0,1)} \sum_{i=0}^N d_i y_i^2,
\end{equation*}
where $d_i$ is the $i$-th element of diagonal matrix $D$.
This allows us to check the distance between the current iterate and the solution of \eqref{eq: quadratic diagonal problem}.
We consider two scenarios of Assumption~\ref{ass: PSG} where the stepsize is constant versus adaptive stepsize in order to show the early stopping in the former case. We give three different plots for each scenario: the optimal value of the objective function, the step length $|x_{n+1}-x_n|$ and the distance to the solution $|x_n -x^*|$. 
\begin{enumerate}
    \item We consider the matrix 
\[
Q=\begin{bmatrix}-2 & 2 & 2\\
2 & 2 & -2\\
2 & -2 & 2
\end{bmatrix},
\]
which has $\left(-4,2,4\right)$ as eigenvalues. Then for all the pair $(a_n^f,u_n^f)\in\partial_{lsc}^\R  f(x_n)$, $a_n^f\geq 4$. We fix the maximum number of iteration $N=101$, the starting point $x_0 = (-5,5,-5)$.
\paragraph{Constant Stepsize:} We take $a_n^f=4$ for all $n\in \N$. The initial value $a_0=200$ and $a_{n+1}=a_n-a_n^f$. The stopping criterion will be
\begin{equation*}
    a_{n+1} \leq a_{n+1}^f-\frac{1}{2\gamma}= 4-\frac{1}{2\gamma}.    
\end{equation*}
We test for multiple values of stepsize $\gamma = 0.01, 0.1, 1, 10$. 
Since $a_n^f$ is positive for all $n\in\N$, the sequence $a_n$ will be decreasing with the lower bound $a_n^f$. This is why we intentionally take a large intial $a_0$, otherwise, if $a_0$ is closed to $a_n^f$ then the algorithm may stop just after several iteration.

Figure \ref{fig:plots1-1} depicts the behaviors of the sequence $(x_n)_{n\in\N}$ and the function values.
As we can see, for all values of $\gamma$, $x_n$ and $f(x_n)$ tend to the optimal solution and optimal values respectively.  Observe that for $\gamma =0.01$ the algorithm takes more time to get to the optimal solution and optimal function value compare to the larger values of $\gamma$. On the other hand, since we fixed the total number of steps $N=101$, the algorithm stops before the number of iteration reach $N$. This behavior supports the observation mentioned in Remark \ref{rmk: const gamma}. 
\begin{figure}
    \centering
    \begin{tabular}{ccc}
        \makecell{\includegraphics[width=0.33\textwidth, trim ={15  0 0 0}, clip]{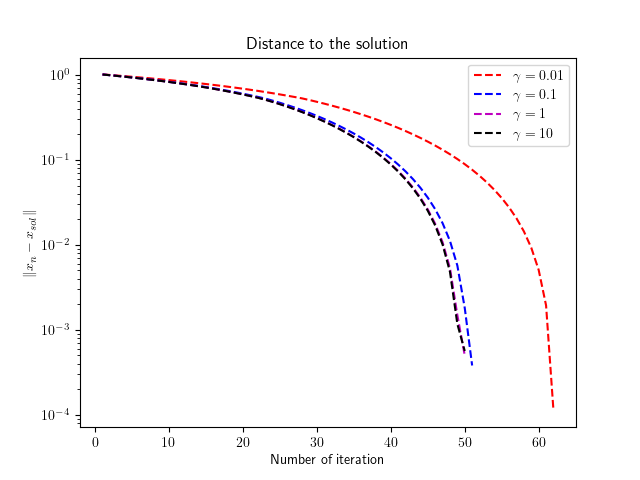}} & 
        \makecell{\includegraphics[width=0.33\textwidth, trim ={15  0 0 0}, clip]{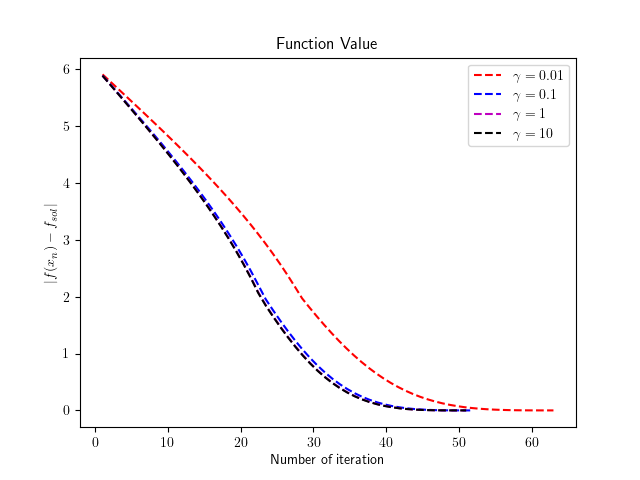}} & \makecell{\includegraphics[width=0.33\textwidth, trim ={15  0 0 0}, clip]{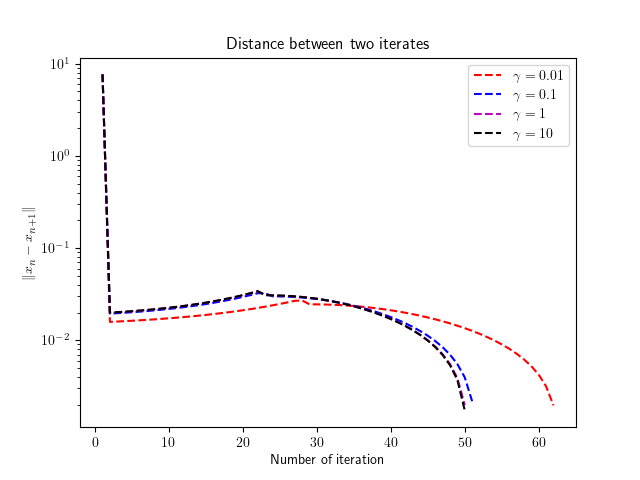}}\\
       \multicolumn{1}{c}{(a)}&  \multicolumn{1}{c}{(b)} & \multicolumn{1}{c}{(c)} 
    \end{tabular}
    \caption{Example~\ref{ex: numerical ball}, case 1, constant stepsize. From left to right: distance between the current iteration and the solution; distance to the optimal value; distance between two successive iterates.}
    \label{fig:plots1-1}
\end{figure}

\paragraph{Adaptive Stepsize:} With the same initial point and the same starting stepsize $\gamma_0$ as in the previous case, we fix $a_n^f=4$ for all $n\in\N$, and set $N=101$. We consider the following 
\begin{equation*}
    (\forall n\in\N)\ a_n = 5 ,\quad \gamma_{n+1} = \frac{\gamma_n}{a_{n+1}}(a_n-4),
\end{equation*}
with the stopping criterion $2\gamma_n (a_n-a_n^f) \leq -1$.
In this scenario, we can both fix $a_n$ and $a_n^f$ for all $n\in\N$. 
At each iteration, the value of the stepsize $\gamma_n$ varies preventing early stopping of the algorithm. The plots are illustrated in Figure \ref{fig:plots1-2}. The algorithm continue to the final iteration $N$ and converges for all starting values of $\gamma_0$. The number of iterations take to arrive at the limit is less than the case with constant stepsize. However, notice that $x_n$ converges to some point which is not the optimal solution of the problem. This happens as we keep $a_n,a_n^f$ to be constants and the starting $\gamma_0$ is small Hence $\gamma_n$ has to go to zero, so $x_n$ tends to stay at the same position. As we allow for bigger $\gamma_0$, the sequence $(x_n)_{n\in\N}$ move toward the solution.
\begin{figure}
    \centering
    \begin{tabular}{ccc}
        \makecell{\includegraphics[width=0.33\textwidth, trim ={15  0 0 0}, clip]{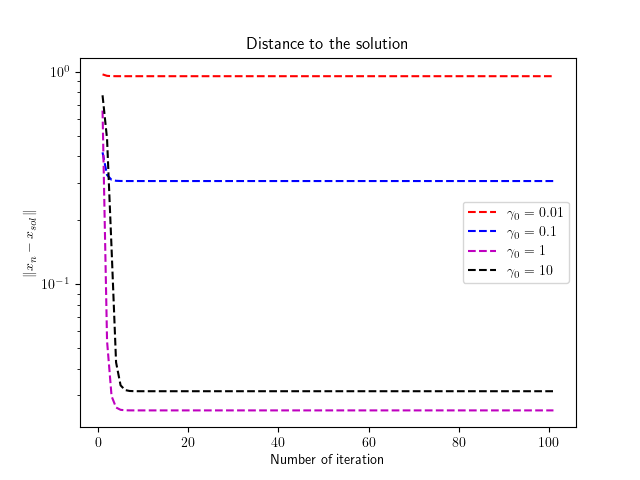}} & 
        \makecell{\includegraphics[width=0.33\textwidth, trim ={15  0 0 0}, clip]{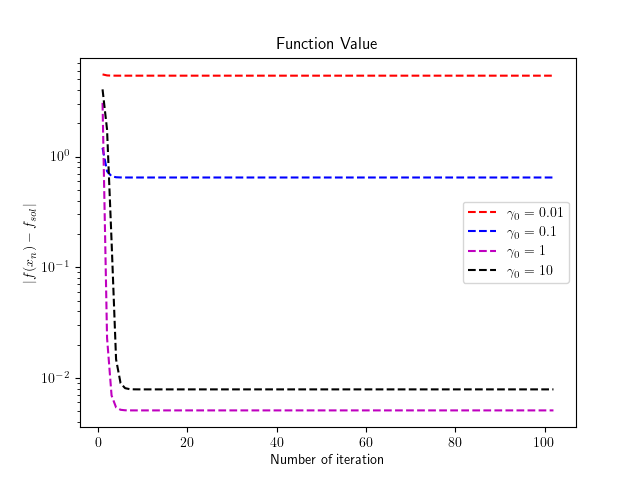}} & \makecell{\includegraphics[width=0.33\textwidth, trim ={15  0 0 0}, clip]{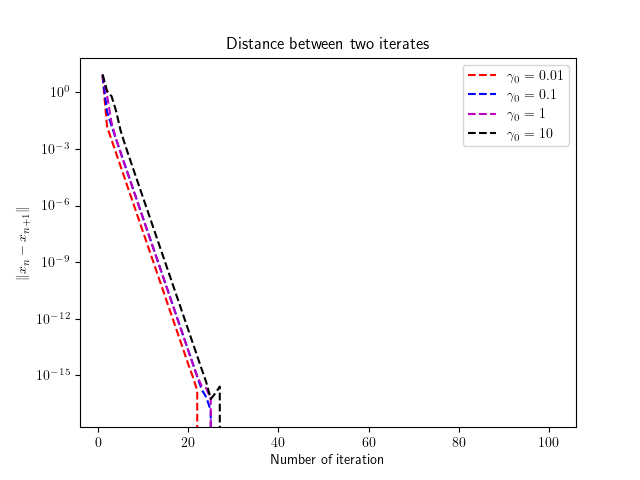}}\\
       \multicolumn{1}{c}{(a)}&  \multicolumn{1}{c}{(b)} & \multicolumn{1}{c}{(c)} 
    \end{tabular}
    \caption{Example~\ref{ex: numerical ball}, case 1, adaptive stepsize. From left to right: distance between the current iteration and the solution; function value at each iteration; distance between two successive iterates.}
    \label{fig:plots1-2}
\end{figure}

\item Now we consider a matrix with two negative eigenvalues, namely
\[
Q=\begin{bmatrix}
1 & 0 & -1 & 1 & 0\\
0 & 1 & 1 & -1 & 0\\
-1 & 1 & -1 & 1 & 1\\
1 & -1 & 1 & -1 & 1\\
0 & 0 & 1 & 1 & 1
\end{bmatrix}
\]
which has $(-3,-1,1,2,2)$ as eigenvalues. The lower bound for $a_n^f$ where $(a_n^f,u_n^f)\in\partial_{lsc}^\R  f(x_n)$, will be $a_n^f\geq 3$. We fix the maximum number of iterations to $N=101$. We also consider constant stepsize and adaptive stepsize with the same starting $\gamma_0$ as in the first case. However, since there are two negative eigenvalues this time, we consider two different starting point $x_{01} = (-10,-10,-10,-10,-10)$ and $x_{02} =(-10,10,-10,10,-10)$. 

\paragraph{Constant Stepsize:} We keep the same setting as in the first case with $a_n^f=3$ for all $n\in \N$. The initial value $a_0=200$ and $a_{n+1}=a_n-a_n^f$. The stopping criterion will be
\begin{equation*}
    a_{n+1} \leq a_{n+1}^f-\frac{1}{2\gamma}= 3-\frac{1}{2\gamma}.    
\end{equation*}
The plots for $x_{01},x_{02}$ are shows in Figure \ref{fig:plots1-201} and Figure \ref{fig:plots1-202} respectively.
\begin{figure}
    \centering
    \begin{tabular}{ccc}
        \makecell{\includegraphics[width=0.33\textwidth, trim ={15  0 0 0}, clip]{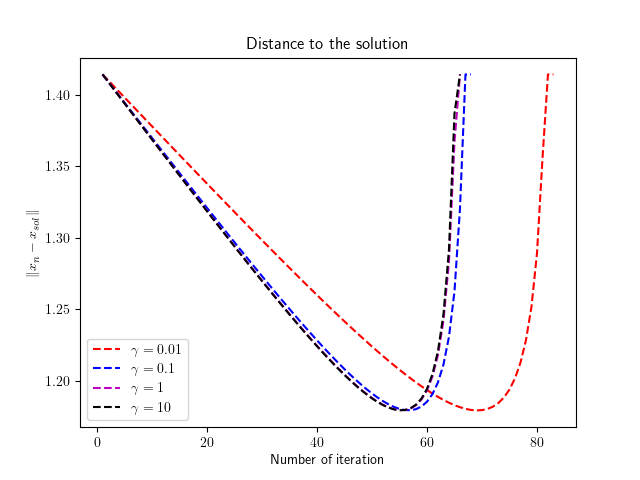}} & 
        \makecell{\includegraphics[width=0.33\textwidth, trim ={15  0 0 0}, clip]{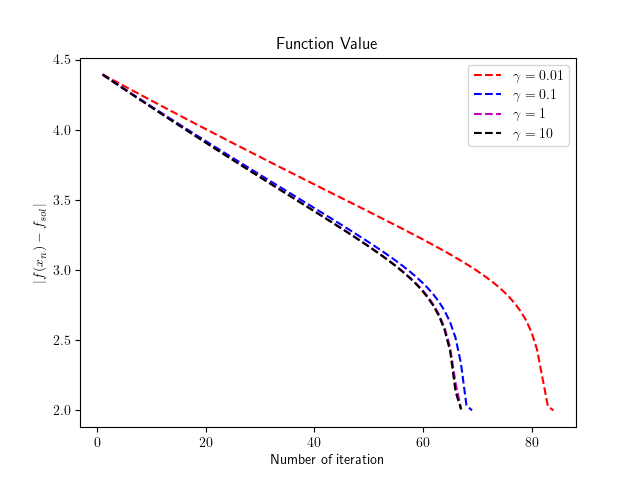}} & \makecell{\includegraphics[width=0.33\textwidth, trim ={15  0 0 0}, clip]{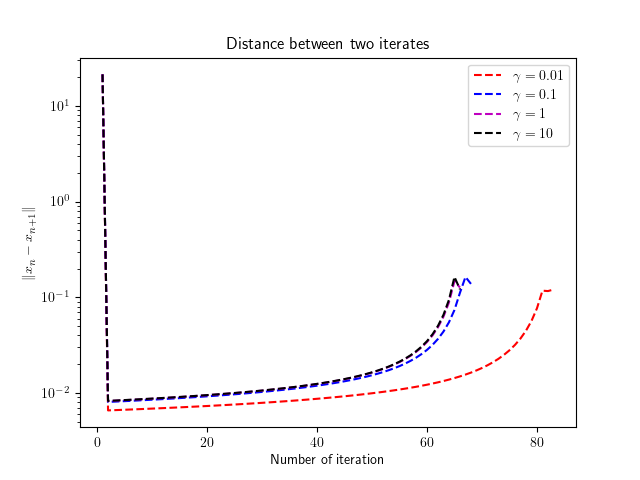}}\\
       \multicolumn{1}{c}{(a)}&  \multicolumn{1}{c}{(b)} & \multicolumn{1}{c}{(c)} 
    \end{tabular}
    \caption{Example~\ref{ex: numerical ball}, case 2, constant stepsize, initialization $x_{01}$. From left to right: distance between the current iteration and the solution; function value at each iteration; distance between two successive iterates.}
    \label{fig:plots1-201}
\end{figure}
\begin{figure}
    \centering
    \begin{tabular}{ccc}
        \makecell{\includegraphics[width=0.33\textwidth, trim ={15  0 0 0}, clip]{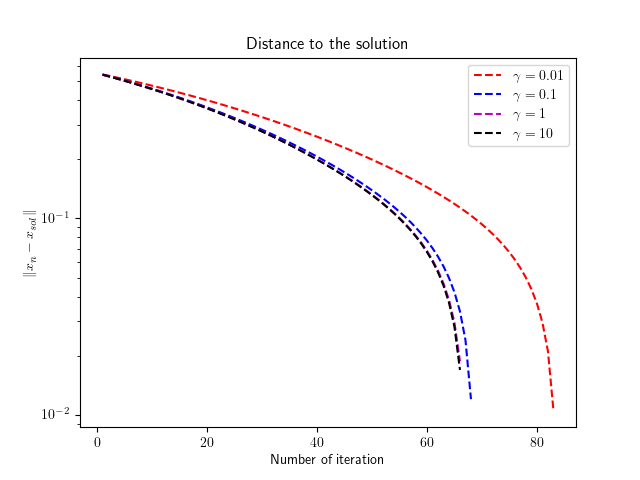}} & 
        \makecell{\includegraphics[width=0.33\textwidth, trim ={15  0 0 0}, clip]{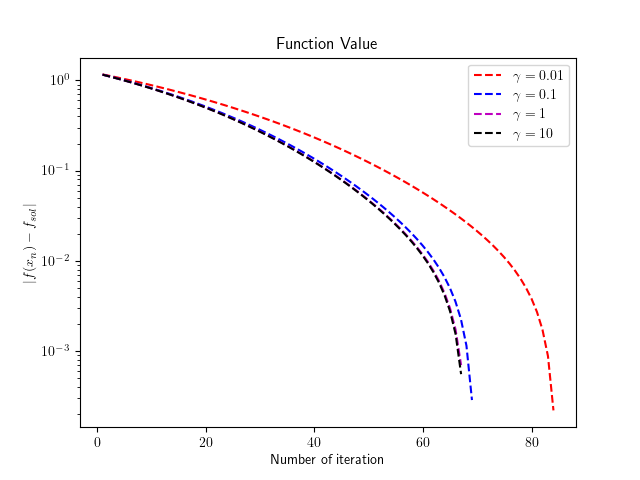}} & \makecell{\includegraphics[width=0.33\textwidth, trim ={15  0 0 0}, clip]{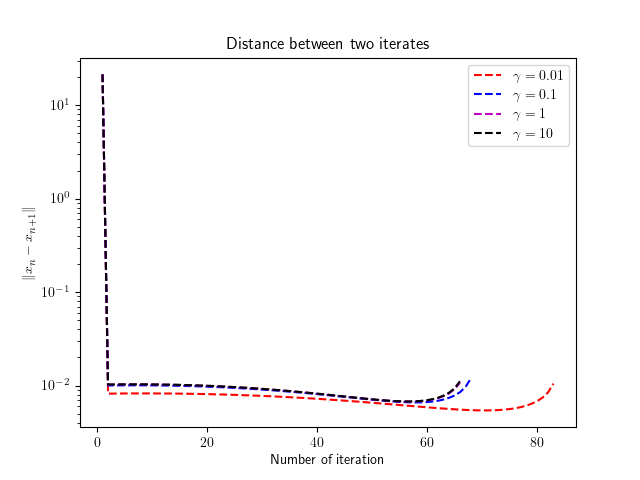}}\\
       \multicolumn{1}{c}{(a)}&  \multicolumn{1}{c}{(b)} & \multicolumn{1}{c}{(c)} 
    \end{tabular}
    \caption{Example~\ref{ex: numerical ball}, case 2, constant stepsize, initialization $x_{02}$. From left to right: distance between the current iteration and the solution; function value at each iteration; distance between two successive iterates.}
    \label{fig:plots1-202}
\end{figure}
Notice that in this case, we have two different negative eigenvalues. The algorithm will converges to the direction of the closest eigenvector with respect to the negative eigenvalue. Here, the starting point $x_{01}$ is actually closer to the eigenvector with respect to eigenvalue $-1$ while $x_{02}$ is closer to the solution with respect to the eigenvalue $-3$ which is our true solution. This explains the different behaviors of the two graphs in Figure \ref{fig:plots1-201} and Figure \ref{fig:plots1-202}.

\paragraph{Adaptive Stepsize:} We fix $a_n^f=3$, maximum iteration $N=101$. Compare to the previous case with adaptive stepsize, we use the following update with $\varepsilon=1$
\begin{equation*}
    a_0 = 4, \ a_{n} = -\frac{1}{2\gamma_n}+a_n^f+\varepsilon ,\quad \gamma_{n+1} = \frac{\gamma_n (a_{n}-a_n^f)+1}{a_n^f +\varepsilon}.
\end{equation*}
This rule ensures that $2\gamma_n (a_n-a_n^f) > -1$ for all $n\in\N$. We only run this case for starting point $x_{02}$, the plots are depicted in Figure \ref{fig:plots1-ada-202}. As both $\gamma_n$ and $a_n$ are changing, the sequence $(x_n)$ converges to the solution for all cases of $\gamma_0$ comparing to Figure \ref{fig:plots1-2}. It is obvious that as the scale of the problem increases, it takes more time for the sequence to converge to the solution.

\begin{figure}
    \centering
    \begin{tabular}{ccc}
        \makecell{\includegraphics[width=0.33\textwidth, trim ={15  0 0 0}, clip]{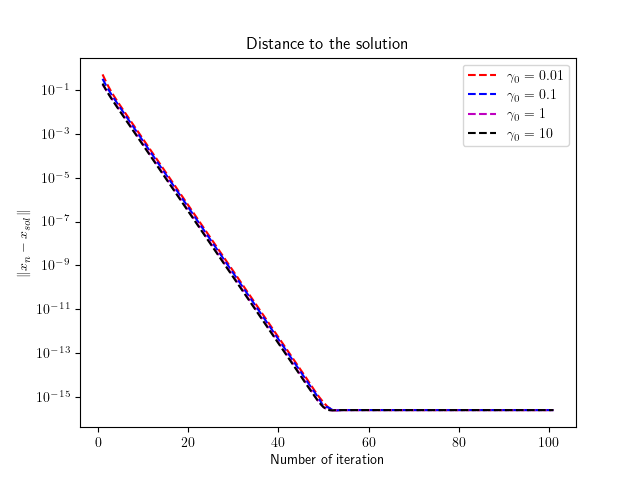}} & 
        \makecell{\includegraphics[width=0.33\textwidth, trim ={15  0 0 0}, clip]{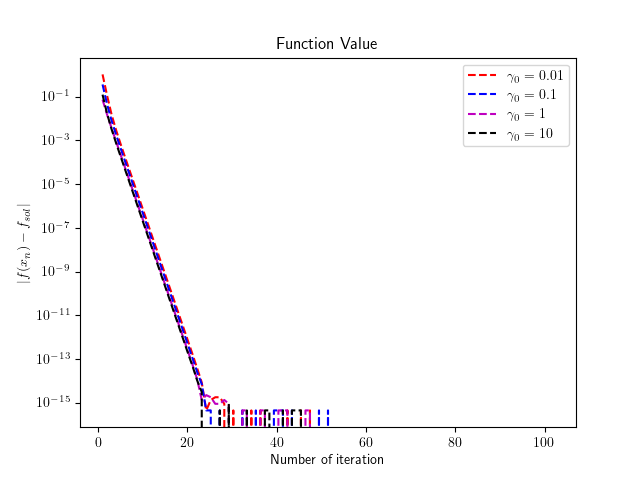}} & \makecell{\includegraphics[width=0.33\textwidth, trim ={15  0 0 0}, clip]{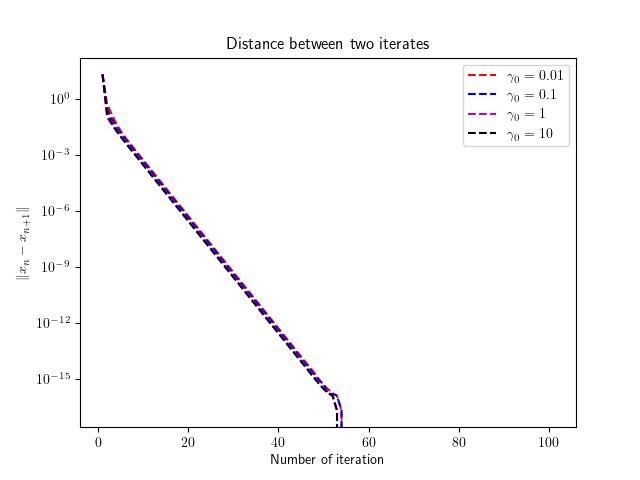}}\\
       \multicolumn{1}{c}{(a)}&  \multicolumn{1}{c}{(b)} & \multicolumn{1}{c}{(c)} 
    \end{tabular}
    \caption{Example~\ref{ex: numerical ball}, case 2, adaptive stepsize, initialization $x_{02}$.From left to right: distance between the current iteration and the solution; function value at each iteration; distance between two successive iterates.}
    \label{fig:plots1-ada-202}
\end{figure}
\end{enumerate}
\end{example}

\begin{example}
\label{ex: negative hessian function}
    Instead of the quadratic function in \eqref{eq: quadratic problem}, we consider a non-quadratic function which has one negative eigenvalue in its Hessian. Consider 
    \[
    f(x,y) = \frac{x^4}{12}+\frac{x^2}{2}-\frac{y^4}{12}-\frac{y^2}{2},
    \]
with the Hessian matrix
\[
\nabla^2 f(x,y) = \begin{bmatrix}
x^2 +1& 0\\
0 & -y^2 -1\\
\end{bmatrix}.
\]
It is difficult to calculate $\lsc^\R$-subgradient of $f$, but we can use Lemma \ref{lem: lsc subdiff and grad} as $f$ has a Lipschitz continuous gradient. We compute  $(a^f,u^f)\in\partial_{lsc}^\R  f(x,y)$ coordinate-wise i.e.
\begin{align*}
    a^f_x =  y^2+1 +\varepsilon, u^f_x = 2 a^f_x x +\nabla_x f(x,y)\\
    a^f_y =  y^2+1 +\varepsilon, u^f_y = 2 a^f_y y +\nabla_y f(x,y).
\end{align*}
As the eigenvalues of the Hessian determines the convexity of $f$, we want to keep $a_x^f,a_y^f$ above the smallest eigenvalues by some $\varepsilon>0$. We use the same update rule for $a_{n+1} = a_n -a_n^f$ and the same stopping criterion. We test for constant stepsizes $\gamma =0.01, 0.1, 1$.

We start the algorithm with initial $a_n=(200,200),(x_0,y_0) = (-5,-1), \varepsilon=0.1$ and max iteration $N=1001$. Here we show only the function value and the distance to the solution in Figure \ref{fig:plots3-hess}.

\begin{figure}
    \centering
    \begin{tabular}{cc}
        \makecell{\includegraphics[width=0.5\textwidth, trim ={15  0 0 0}, clip]{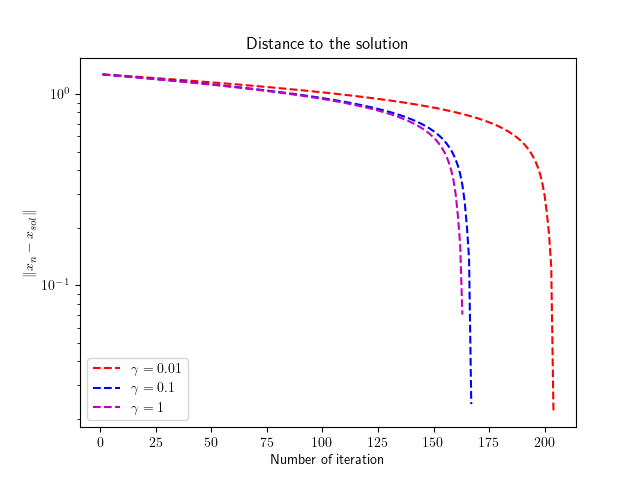}} & 
        \makecell{\includegraphics[width=0.5\textwidth, trim ={15  0 0 0}, clip]{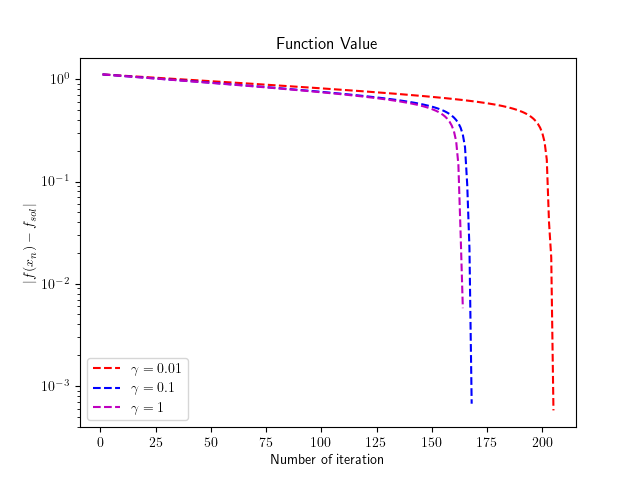}} \\
       \multicolumn{1}{c}{(a)}&  \multicolumn{1}{c}{(b)}
    \end{tabular}
    \caption{$\lsc^\R$-(PSG) for Example \ref{ex: negative hessian function}. From left to right: distance between the current iterate and the solution; function value at each iteration.}
    \label{fig:plots3-hess}
\end{figure}
\end{example}
\paragraph{Funding}{This work was funded by the European Union's Horizon 2020 research and innovation program under the Marie Sk{\l}odowska-Curie grant agreement No. 861137. This work represents only the authors' view and the European Commission is not responsible for any use that may be made of the information it contains.}

\bibliographystyle{acm}
\bibliography{ref_library}
\end{document}